\documentclass{amsart}
\usepackage{amssymb}
\usepackage{enumerate}
\usepackage{adjustbox}
\usepackage{enumitem}
\usepackage{tikz-cd}
\usepackage{hyperref}
\usepackage[normalem]{ulem}

\numberwithin{equation}{section}
\theoremstyle{plain}

\newtheorem*{theorem*}{Theorem}
\newtheorem{theorem}[equation]{Theorem} 

\newtheorem{corollary}[equation]{Corollary} 
\newtheorem{lemma}[equation]{Lemma}
\newtheorem{proposition}[equation]{Proposition}

\theoremstyle{definition}
\newtheorem{definition}[equation]{Definition} 
 
\newtheorem{example}[equation]{Example}
\newtheorem{hypothesis}[equation]{Hypothesis} 

\newtheorem{remark}[equation]{Remark}

\DeclareMathOperator\Aut{Aut}

\DeclareMathOperator\Gr{GrMod}

\DeclareMathOperator\id{id}

\DeclareMathOperator\reg{reg}

\DeclareMathOperator\Span{Span}

\newcommand\CC{\mathbb C}
\newcommand\NN{\mathbb N}

\newcommand\ZZ{\mathbb Z}

\newcommand\cA{\mathcal A}

\newcommand\cC{\mathcal C}

\newcommand\cI{\mathcal I}
\newcommand\cJ{\mathcal J}

\newcommand\fs{\mathfrak s}

\newcommand\be{\mathbf e}

\newcommand\bu{\mathbf u}
\newcommand\bv{\mathbf v}

\newcommand\bzero{\mathbf 0}

\newcommand\kk{\Bbbk}

\newcommand\cnt{\mathcal Z}
\renewcommand{\int}{\mathrm{int}}
\newcommand\inv{^{-1}}

\newcommand\iso{\cong}

\newcommand\tensor{\otimes}

\newcommand\restrict[1]{\raisebox{-.3ex}{$|$}_{#1}}

\renewcommand{\to}{\ensuremath{\longrightarrow}}

\usepackage{xcolor}

\newcommand\gt{\chi}
\newcommand\lt{{}^\gt\!}
\newcommand\ltA{{}^\gt\! A}
\newcommand\ltC{{}^\gt C}

\newcommand\ct{\mathbf c}

\title{Twists of twisted generalized Weyl algebras}
\author[Gaddis]{Jason Gaddis}
\address{Department of Mathematics, Miami University, Oxford, Ohio 45056} 
\email{gaddisj@miamioh.edu}

\author[Rosso]{Daniele Rosso}
\address{Department of Mathematics and Actuarial Science, Indiana University Northwest, Gary, IN 46408} 
\email{drosso@iu.edu}

\makeatletter
\@namedef{subjclassname@2020}{\textup{2020} Mathematics Subject Classification}
\makeatother

\subjclass[2020]{
16W50, %Graded rings and modules
16S32, %Rings of differential operators (associative algebraic aspects)
16W22, %Actions of groups and semigroups; invariant theory (associative rings and algebras)
16D90 %Module categories in associative algebras [See also 16Gxx, 16S90]; module theory in a category-theoretic context; Morita equivalence and duality
}
\keywords{Twisted generalized Weyl algebra, twisted tensor product, graded twist, cocycle twist, graded modules}
\begin{document}

\begin{abstract}
We study graded twisted tensor products and graded twists of twisted generalized Weyl algebras (TGWAs). We show that the
class of TGWAs is closed under these operations assuming mild hypotheses. We generalize a result on cocycle equivalence amongst multiparameter quantized Weyl algebras to the setting of TGWAs.
As another application we prove that certain TGWAs of type $A_2$ are noetherian.
\end{abstract}

\maketitle

\section{Introduction}

The Weyl algebras $A_n(\CC)$ are fundamental objects in noncommutative algebra. It is well-known that $A_2(\CC) \iso A_1(\CC) \tensor A_1(\CC)$. More generally, if $A$ and $B$ are two rank one generalized Weyl algebras (GWAs), then $A \tensor B$ is a rank two GWA. 
There is a notion of a \emph{twisted generalized Weyl algebras} (TGWAs), 
%due to Mazorchuk and Turowska \cite{MT}, 
and under suitable hypotheses, the tensor product of two TGWAs is again a TGWA \cite{GR}. In this work we extend the above results to the case of graded \emph{twisted tensor products}. The definition of this is given in Section \ref{sec.twistedtensor}.

\begin{theorem*}[Theorem \ref{thm.twtensor}]
Let $A$ and $B$ be regular, consistent TGWAs. 
%of rank $m$ and $n$, respectively. 
Let $\tau:B \tensor A \to A \tensor B$ be a twisting map satisfying certain conditions 
(see Hypothesis \ref{hyp.tgwa} and \eqref{eq.rhosigma-2}).
%(see Hypothesis \ref{hyp.twtensor}) 
%and respecting the $\ZZ^m$ and $\ZZ^n$-gradings on $A$ and $B$, respectively. 
Then $A \tensor_\tau B$ is a regular, consistent TGWA. 
%of rank $m+n$.
\end{theorem*}

A partial converse to Theorem \ref{thm.twtensor} is given in Theorem \ref{thm.converse}. 
In \cite{GRW2}, the authors along with Robert Won consider twisted tensor products of BR algebras \cite{BR,GRW1}. These results have mild overlap in the case of TGWAs of Cartan type $(A_1)^n$.

We also consider twists of TGWAs by graded twisting systems. 
%as defined by Zhang \cite{ztwist}. 
A definition is given in Definition \ref{def.gr-twist}, along with that of the graded twist $\ltA$ of a graded algebra $A$. We show that the class of TGWAs is invariant under graded twisting.

\begin{theorem*}[Theorem \ref{thm.TGWAtwist}]
Let $A=A_\mu(R,\sigma,t)$ be a TGWA 
%of rank $n$ 
and $\gt$ a twisting system for $A$ that respects the $\ZZ^n$-grading on $A$. 
Then $\ltA$ is again a TGWA. 
%of rank $n$.
\end{theorem*}

When considering a graded twist of the tensor product of two TGWAs, the two notions coincide under certain conditions. See Proposition \ref{prop.twtw}.

A GWA over a noetherian base ring is again noetherian. However, whether this holds for TGWAs is an open question. In \cite{GR}, the authors proved this for TGWAs of Cartan type $(A_1)^n$ as well as certain TGWAs of Cartan type $A_2$. Here we obtain additional rank 2 examples by graded twisting (Example \ref{ex.A2}).

%\begin{theorem}[Theorem \ref{thm.twist}]
%Let $A=A_\mu(R,\sigma,t)$ and $A'=A_\nu(R,\sigma,t)$ be $\kk$-finitistic TGWAs that are completely determined by Serre-type relations. Then $A'$ is a graded twist of $A$ if and only if $\nu_{ij}\nu_{ji}=\mu_{ij}\mu_{ji}$ for all $i \neq j$.
%\end{theorem}

%We show that every $\ZZ^n$-graded twisted tensor product of two TGWAs can be realized as a graded twist of the ordinary tensor product. Both approaches have their advantages, however. For example, as mentioned above, the twisted tensor product of two TGWAs is again a TGWA under suitable hypothesis. We do not have such a result for graded twists. However, graded twists preserve the noetherian property and so we are able to realize further examples of noetherian TGWAs, building on work in \cite{GR}.

\section{Background}

Throughout $\kk$ is a field and we denote by $\kk^\times$ the group of units of $\kk$. All tensor products should be assumed to be over $\kk$. For each $n\in \NN$, we let $[n]=\{1,2,\ldots,n\}$.

All algebras will be considered as unital $\kk$-algebras. 
For an algebra $A$, $\cnt(A)$ denotes the center of $A$ and $A_{\reg}$ the set of regular elements of $A$.
%We denote the center of an algebra $A$ by $\cnt(A)$. Let $A_{\reg}$ denote the set of regular elements of $A$.  
%Unless stated otherwise, all modules are left modules.
The algebra $A$ is \emph{graded} by a monoid $\Gamma$ if there is a vector space decomposition $A = \bigoplus_{\gamma \in \Gamma} A_\gamma$ satisfying $A_\alpha \cdot A_\beta \subset A_{\alpha\beta}$ for all $\alpha,\beta \in \Gamma$. A map $\psi:A \to B$ of $\Gamma$-graded algebras is \emph{graded} if $\psi(A_\gamma) \subset B_\gamma$ for all $\gamma \in \Gamma$.

We begin by discussing our main objects of interest (TGWAs). Subsequently we will define twisted tensor products, graded twists, and cocycle twists.

\subsection{TGWAs}

The definition of a TGWA is due to Mazorchuk and Turowska \cite{MT}. This construction, in particular, realizes all generalized Weyl algebras.

%Through this construction, one can realize all generalized Weyl algebras, including the classical Weyl algebras, and primitive quotients of enveloping algebras of finite-dimensional complex simple Lie algebras \cite{HVS}.

\begin{definition}\label{defn.tgwa}
Let $n$ be a positive integer. A \emph{twisted generalized Weyl datum (TGWD) of rank $n$} is the triple $(R,\sigma,t)$ where $R$ is an algebra, $\sigma=(\sigma_1,\hdots,\sigma_n)$ is an $n$-tuple of commuting automorphisms of $R$, and $t=(t_1,\hdots,t_n) \in \cnt(R)^n$.
%is an $n$-tuple of nonzero central elements of $R$. 
Given such a TGWD and $\mu=(\mu_{ij}) \in M_n(\kk^\times)$, the associated \emph{twisted generalized Weyl construction (TGWC)}, $\cC_\mu(R,\sigma,t)$, is the algebra generated over $R$ by the $2n$ indeterminates $X_1^{\pm},\hdots,X_n^{\pm}$ subject to the relations
\[ 
X_i^{\pm} r - \sigma_i^{\pm 1}(r) X_i^{\pm}, \quad
X_i^-X_i^+ - t_i, \quad X_i^+X_i^- - \sigma_i(t_i), \quad
X_i^+ X_j^- - \mu_{ij} X_j^- X_i^+.
\]
%\begin{align}
%\label{eq.tgwc1} &X_i^{\pm} r - \sigma_i^{\pm 1}(r) X_i^{\pm} & &\text{for all $r \in R$ and $i \in [n]$,} \\
%\label{eq.tgwc2} &X_i^-X_i^+ - t_i, \qquad X_i^+X_i^- - \sigma_i(t_i) &	&\text{for all $i \in [n]$,} \\
%\label{eq.tgwc3} &X_i^+ X_j^- - \mu_{ij} X_j^- X_i^+ & &\text{for all $i,j \in [n]$, $i \neq j$.}
%\end{align}
for all $r \in R$ and $i,j \in [n]$, $i \neq j$.
There is a $\ZZ^n$-grading on $\cC_\mu(R,\sigma,t)$ obtained by setting $\deg(r)=\bzero$ for all $r \in R$ and $\deg(X_i^{\pm})=\pm \be_i$ for all $i \in [n]$. The associated \emph{twisted generalized Weyl algebra (TGWA)}
is $\cA_\mu(R,\sigma,t)=\cC_\mu(R,\sigma,t)/\cJ$ 
%, $A=\cA_\mu(R,\sigma,t)$, is the quotient $\cC_\mu(R,\sigma,t)/\cJ$ 
where $\cJ$ is the sum of all graded ideals $J = \bigoplus_{g \in \ZZ^n} J_{g}$ such that $J_\bzero = \{0\}$.
\end{definition}

%When $\mu_{ij}=1$ for all $i \neq j$, we simplify the above notation by omitting $\mu$. 

%and let $\cnt(R)$ denote the center of $R$. 

\begin{definition}
The TGWD $(R,\sigma,t)$ is \emph{regular} if $t_i \in R_{\reg}$ for all $i \in [n]$. For a parameter matrix $\mu$, $(R,\sigma,t)$ is \emph{$\mu$-consistent} if the canonical map $R \to A_\mu(R,\sigma,t)$ is injective. The TGWA $A_\mu(R,\sigma,t)$ is regular (resp. $\mu$-consistent) if the underlying TGWD is regular (resp. $\mu$-consistent).
\end{definition}

Suppose $(R,\sigma,t)$ is a regular TGWD. 
By \cite{hart4}, the ideal $\cJ$ in Definition \ref{defn.tgwa} can also be defined as follows:
\begin{align}\label{eq.ideal}
\cJ = \{ c \in C : rc=0 \text{ for some } r \in R_{\reg} \cap \cnt(R) \}.
\end{align}
By \cite[Theorem 6.2]{FH2}, $(R,\sigma,t)$ is $\mu$-consistent if and only if the following hold:
%equations hold:
\begin{align}
\label{eq.cons1} &\sigma_i\sigma_j(t_it_j)=\mu_{ij}\mu_{ji}\sigma_i(t_i)\sigma_j(t_j) \text{ for all distinct $i,j \in [n]$,} \\
\label{eq.cons2} &t_j\sigma_i\sigma_k(t_j) = \sigma_i(t_j)\sigma_k(t_j) \text{ for all pairwise distinct $i,j,k \in [n]$.}
\end{align}

Let $A=\cA_\mu(R,\sigma,t)$ be a regular, $\mu$-consistent TGWA. For $i,j \in [n]$, define
\begin{align}\label{eq.vij}
    V_{ij} = \Span_\kk\{\sigma_i^k(t_j) ~|~ k \in \ZZ\}.
\end{align}
Then $A$ is \emph{$\kk$-finitistic} if $\dim_\kk V_{ij} < \infty$ for all $i,j$. 
%We note that the regularity and consistency assumption are not a standard part of the definition of $\kk$-finitistic TGWAs, but are necessary for significant results in the literature. 
In this case,
%Let $A=\cA_\mu(R,\sigma,t)$ be $\kk$-finitistic, 
let $p_{ij} \in \kk[x]$ denote the minimal polynomial for $\sigma_i$ acting on $V_{ij}$, and let $C_A = (a_{ij})$ where $a_{ii}=2$ and $a_{ij}=1-\deg p_{ij}$ for $i \neq j$.
We say $A$ is of type $T$ when $C_A$ is a (generalized) Cartan matrix of type $T$.

%Of particular interest in this paper will be TGWAs of type $(A_1)^n$. For all $i \neq j$, $p_{ij}(x)=x-\gamma_{ij}$ for some $\gamma_{ij} \in \kk^\times$. This condition is equivalent to 
%\begin{align}\label{eq.sigma}
%\sigma_i(t_j) = \gamma_{ij} t_j \quad\text{for all $i \neq j$.}
%\end{align}
%In this case, the consistency equations \eqref{eq.cons1}-\eqref{eq.cons2} are equivalent to
%\begin{align}\label{eq.Acons}
%\mu_{ij}\mu_{ji} = \gamma_{ij}\gamma_{ji} \text{ for all $i \neq j$}.
%\end{align}
%These TGWAs have an explicit presentation as given in the next result.

%\begin{theorem}[Futorny, Hartwig {\cite[Theorem 4.1]{FH1}}]
%\label{thm.A1n}
%Let $A=A_\mu(R,\sigma,t)$ be a TGWA of type $(A_1)^n$ with the $\gamma_{ij}$ as above. Then $A$ is (isomorphic to) the $\kk$-algebra generated over $R$ by $X_1^{\pm},\hdots,X_n^{\pm}$ with relations
%\begin{align*}
%&X_i^{\pm} r - \sigma_i^{\pm 1}(r) X_i^{\pm}, \quad
%X_i^-X_i^+ - t_i, \quad X_i^+X_i^- - \sigma_i(t_i), \\
%&X_i^+ X_j^- - \mu_{ij} X_j^- X_i^+, \quad
%X_i^+ X_j^+ - \gamma_{ij}\mu_{ij}\inv X_j^+ X_i^+, \quad
%X_j^- X_i^- - \gamma_{ij} \mu_{ji}\inv X_i^- X_j^-,
%\end{align*}
%for all $r \in R$ and $i,j \in [n]$, $i \neq j$.
%\end{theorem}

%Generalized Weyl algebras are exactly the TGWAs of type $(A_1)^n$ in which $\gamma_{ij}=\mu_{ij}=1$ for all $i\neq j$.

\subsection{Twisted tensor products}
\label{sec.twistedtensor}

We recall the definition of a twisted tensor product, due to \u{C}ap, Schichl, and Van\u{z}ura \cite{cap}.

\begin{definition}\label{defn.twtensor}
A \emph{twisted tensor product} of algebras $A$ and $B$ is a triple $(C,\iota_A,\iota_B)$ where $C$ is an algebra, and $\iota_A:A \to C$ and $\iota_B:B \to C$ are injective (algebra) homomorphisms such that
the map $a \tensor b \mapsto \iota_A(a) \cdot \iota_B(b)$
%\begin{align*}
%(\iota_A,\iota_B):A \tensor B &\to C \\
%a \tensor b &\mapsto \iota_A(a) \cdot \iota_B(b)
%\end{align*} 
is a linear isomorphism.
\end{definition}

%We will be interested in twisted tensor products arising through a particular construction. 
A \emph{twisting map} for $(A,B)$ is a $\kk$-linear map $\tau:B \tensor A \to A \tensor B$ such that $\tau(b \tensor 1)=1 \tensor b$ and $\tau(1 \tensor a) = a \tensor 1$. Let $\mu_A$ and $\mu_B$ be the multiplication maps on $A$ and $B$, respectively. 
Define $C$ to be $A \tensor B$ (as a vector space) with multiplication 
\[ \mu_\tau : = (\mu_A \tensor \mu_B) \circ (\id_A \tensor \tau \tensor \id_B).\] 
Define $\iota_A:A \to A \tensor B$ by $\iota_A(a) = a \tensor 1$ and $\iota_B$ is defined similarly. Then $(C,\iota_A,\iota_B)$ is a twisted tensor product of $A$ and $B$ if and only if $\mu_\tau$ is associative. By \cite[Proposition 2.3]{cap}, this holds if and only if
the following diagram commutes
\[ 
\adjustbox{scale=.95,center}{%
\begin{tikzcd}
& B \tensor A \tensor B \tensor A \ar[r, "\tau \tensor \tau"] 
    & A \tensor B \tensor A \tensor B \ar[dr,"\id_A \tensor \tau \tensor \id_B"] \\
B \tensor B \tensor A \tensor A \arrow[ur, "\id_B \tensor \tau \tensor \id_A"]
    \arrow[r, swap, "\mu_B \tensor \mu_A"]
& B \tensor A \arrow[r,swap,"\tau"]
& A \tensor B & A \tensor A \tensor B \tensor B
    \arrow[l, "\mu_A \tensor \mu_B"]
\end{tikzcd}}
\]
When this holds, we simplify the notation $(C,\iota_A,\iota_B)$ to $A \tensor_\tau B$.
By \cite[Proposition 2.7]{cap}, all twisted tensor products can be realized, up to isomorphism, through a twisting map.

\begin{example}\label{ex.ttps}
(1) Given any two algebras $A$ and $B$, one can take the trivial twisting map defined by $\tau(b \tensor a)=a \tensor b$ for all $b \tensor a \in B \tensor A$. The corresponding twisted tensor product $A \tensor_\tau B$ is then just the usual tensor product $A \tensor B$.

(2) Let $A=\kk[x]$, $B=\kk[y]$, and $q\in \kk^\times$. Let $\tau:B \tensor A \to A \tensor B$ be the linear map defined by $\tau(y^n \tensor x^m) = q^{mn}(x^m \tensor y^n)$ for $m,n \in \NN$. Then $\tau$ is a twisting map and 
$A \tensor_\tau B \iso \kk_q[x,y] = \kk\langle x,y \mid yx-qxy\rangle$.

(3) \cite[Section 1.5]{AJ} Let $\Lambda=(\lambda_{ij}) \in M_n(\kk^\times)$ satisfying $\lambda_{ii}=1$ for all $i$ and $\lambda_{ij}=\lambda_{ji}\inv$ for $i \neq j$.
%be a multiplicatively antisymmetric matrix (so $\lambda_{ii}=1$ for all $i$ and $\lambda_{ij}=\lambda_{ji}\inv$ for $i \neq j$). 
Let $q=(q_i) \in (\kk^\times)^n$. The \emph{alternative quantized Weyl algebra} is the algebra generated by $x_1,y_1,\hdots,x_n,y_n$ with relations
\[
    x_ix_j - \lambda_{ij} x_jx_i, ~~ 
    x_iy_j - \lambda_{ji} y_jx_i, ~~
    y_jy_i - \lambda_{ji}y_iy_j, ~~
    x_jy_i - \lambda_{ij} y_ix_j, ~~
    x_jy_j - q_j y_jx_j - 1,
\]
for all $i,j \in [n]$.
%\begin{align*}
%    x_ix_j &= \lambda_{ij} x_jx_i, & y_jy_i = \lambda_{ji}y_iy_j, \\
%    x_iy_j &= \lambda_{ji} y_jx_i, & x_jy_i = \lambda_{ij} y_ix_j, \\
%    \mathrlap{x_jy_j - q_j y_jx_j = 1.}
%\end{align*}
We can realize this algebra as an iterated twisted tensor product over copies of the first quantum Weyl algebra. See Example \ref{ex.mpwa}.
\end{example}

Suppose that $A$ is $\ZZ^m$-graded and that $B$ is $\ZZ^n$-graded. We say that a twisting map $\tau:B \tensor A \to A \tensor B$ is \emph{graded} if $\tau(B_\beta \tensor A_\alpha) \subset A_\alpha \tensor B_\beta$ for all $\alpha \in \ZZ^m$ and $\beta \in \ZZ^n$.
That is, $\tau$ respects the gradings on $A$ and $B$. Given a graded twisting map $\tau$, the twisted tensor product $A \tensor_\tau B$ is then (naturally) $\ZZ^{m+n}$-graded.

\subsection{Graded twists}

The notion of a graded twist is integral to the study of noncommutative projective algebraic geometry. Our treatment follows Zhang \cite{ztwist}.
%Our treatment will follow the paper of Zhang \cite{ztwist}. 

%\begin{definition}
%Let $\Gamma$ be a semigroup and $A$ be a $\Gamma$-graded algebra. A set of algebra automorphisms $\gt=\{ \gt_\alpha \mid \alpha \in \Gamma\}$ of $A$ is a \emph{left twisting system} if
%$\gt_\alpha(\gt_\beta(a)b)=\gt_{\alpha\beta}(a)\gt_\alpha(b)$
%for all $\alpha,\beta \in \Gamma$, $a \in A_\alpha$, $b \in A_\beta$. Given a right graded twisting system $\gt$ of $A$, the \emph{graded twist of $A$ by $\gt$} to be the algebra $A^\gt$ that is $A$ as a $\kk$-vector space but with multiplication $\star$ given by $a \star b = \gt_\beta(a)b$ for $a \in A_\alpha$ and $b \in A_\beta$.
%\end{definition}

\begin{definition}\label{def.gr-twist}
Let $\Gamma$ be a monoid and let $A$ be a $\Gamma$-graded algebra. A set of graded $\kk$-linear bijections of $A$, $\gt=\{ \gt_\alpha \mid \alpha \in \Gamma\}$
%\jason{$\gt_\alpha(A_\beta)\subseteq A_\beta$ for all $\alpha,\beta\in\Gamma$ (cut?)}, 
is a \emph{left twisting system} if
\begin{equation}\label{eq:left-twist}\gt_\alpha(\gt_\beta(c)b)=\gt_{\beta\alpha}(c)\gt_\alpha(b)\text{ for all }\alpha,\beta,\gamma \in \Gamma,~b \in A_\beta,~c\in A_\gamma.\end{equation} 
Given a left twisting system $\gt$ of $A$, the \emph{left graded twist of $A$ by $\gt$} is the algebra $\ltA$ that is $A$ as a $\kk$-vector space but with multiplication $\star$ given by $a \star b = \gt_\beta(a)b$ for $a \in A_\alpha$ and $b \in A_\beta$.
\end{definition}

Condition \eqref{eq:left-twist} guarantees that the product in $\ltA$ is associative. (By \cite[Prop 2.5]{ztwist}, graded twisting gives an equivalence relation on the set of $\Gamma$-graded algebras.)
Similarly, the condition $\gt_\alpha(b \gt_ \beta(c)) = \gt_\alpha(b) \gt_{\alpha\beta}(c)$ defines a \emph{right twisting system} and there is an associative multiplication $a \star b =a\gt_\alpha(b)$ for the \emph{right graded twist of $A$ by $\gt$}, denoted by $A^\gt$.

\begin{remark}\label{rem:twist-system}
Condition \eqref{eq:left-twist} is satisfied if all the maps $\gt_\alpha$ are algebra automorphisms of $A$ and the map $\Gamma^{\operatorname{op}}\to \Aut_\kk(A)$, given by $\alpha\to \gt_\alpha$, is a homomorphism. We focus on this setting and we take $\Gamma=\ZZ^n$. In this case we have $\gt_0=\id_A$ and hence $1_A=1_{\ltA}$ (see \cite[\S 2]{ztwist}).
\end{remark}

\subsection{Cocycle twists}
\label{sec.cocycle}

Let $\Gamma$ be a group. A \emph{$\kk^\times$-valued 2-cocycle on $\Gamma$} is a map $\ct: \Gamma \times \Gamma \to \kk^\times$ such that
\begin{align}\label{eq:cocycle} 
\ct(\alpha,\beta\gamma)\ct(\beta,\gamma) = \ct(\alpha\beta,\gamma)\ct(\alpha,\beta) \quad\text{for all $\alpha,\beta,\gamma \in \Gamma$.}
\end{align}
%for all $\alpha,\beta,\gamma \in \Gamma$. 
The set of all $\kk^\times$-valued 2-cocycles on $\Gamma$ is denoted $Z^2(\Gamma,\kk^\times)$. Cocycle twists provide an alternative way to construct twists of graded algebras and are common in the quantum groups literature (see, e.g., \cite{ASTgln}).
The cocycle condition guarantees that the twisted algebra $A^\ct$, defined below, is associative.

\begin{definition}
Let $A$ be a $\Gamma$-graded algebra and $\ct \in Z^2(\Gamma,\kk^\times)$. The \emph{cocycle twist of $A$ by $\ct$} is the algebra $A^\ct$ which is $A$ as a $\kk$-vector space but with multiplication $*_{\ct}$ given by $a*_{\ct} b = \ct(\alpha,\beta)ab$ for all $a \in A_\alpha$, $b \in A_\beta$, $\alpha,\beta \in \Gamma$.
\end{definition}

If $\Gamma$ is a group, $\ct\in Z^2(\Gamma,\kk^\times)$, and $A$ is a $\Gamma$-graded algebra, then for each $\alpha\in\Gamma$ we can define a graded $\kk$-linear map by $\gt_\alpha(x)=\ct(\beta,\alpha)x$, where $x\in A_\beta$. It is immediate to check that the cocyle condition \eqref{eq:cocycle} implies the twisting system condition \eqref{eq:left-twist} and that $\ltA\simeq A^\ct$ as algebras, with the isomorphism being the identity map on $A$ (see, e.g., \cite[Thm 2.2.19]{davies}). Similarly, we could define a right twisting system by $\gt'_\alpha(x)=\ct(\alpha, \beta)x$, with $x\in A_\beta$, and have $A^{\gt'}\simeq A^{\ct}$. Hence, cocycle twists can be considered as special cases of graded twists.
%and we will only work with the latter in this article.

\section{Twisted tensor products of TGWAs}\label{sec.twist-tensor}

In this section, we show that the twisted tensor product of two TGWAs is again a TGWA under mild hypotheses. Our proof, including the following prepartory lemma, closely follows 
%First we need a preparatory lemma which follows from the proof of 
\cite[Theorem 2.16]{GR}.

\begin{lemma}\label{lem.reg}
Let $R$ and $S$ be algebras. Suppose that $t \in R$ and $h \in S$ are regular. Then $t \tensor 1$ and $1 \tensor h$ are regular in $R \tensor S$.
%Suppose that $t_j \in R$ and $h_i \in S$ are regular, $1 \leq j \leq m$, $1 \leq i \leq n$. Then the elements $t_j \tensor 1$ and $1 \tensor h_i$ are regular in $R \tensor S$.
\end{lemma}

We assume the following hypothesis throughout this section.

\begin{hypothesis}\label{hyp.tgwa}
Assume that $R$ and $S$ are algebras. Let $C=\cC_\mu(R,\sigma,t)$ be a regular, $\mu$-consistent TGWC of rank $m$ with canonical generators $\{X_j^{\pm}\}$, $\cJ$ the canonical ideal associated to $C$, and $A=C/\cJ$ the resulting TGWA.
Let $D=\cC_\nu(S,\rho,h)$ be a regular, $\nu$-consistent TGWC of rank $n$ with canonical generators $\{Y_i^{\pm}\}$, $\cI$ the canonical ideal associated to $D$, and $B=D/\cI$ the resulting TGWA. 
We extend the $\sigma_j$ and the $\rho_i$ to commuting automorphisms of $R \tensor S$ such that $\rho_i(R \tensor 1) \subset R \tensor 1$ for each $i$ and $\sigma_j(1 \tensor S) \subset 1 \tensor S$ for each $j$.
\end{hypothesis}

Hypothesis \ref{hyp.tgwa} implies that each $\rho_i$ (resp. $\sigma_j$) may be viewed as an automorphism of $R$ (resp. $S$), which, by an abuse of notation, we also call $\rho_i$ (resp. $\sigma_j$). One can always extend the automorphisms in a trivial way (e.g., $\sigma_j\restrict{S} = \id_S$.)

%Hypothesis \ref{hyp.tgwa} implies that each $\rho_i$ may be viewed as an automorphism of $R$, which, by an abuse of notation, we also call $\rho_i$. Similarly, we may view each $\sigma_j$ as an automorphism of $S$. We note that one can always extend the automorphisms in a trivial way (e.g., $\sigma_j\restrict{S} = \id_S$.)

To simplify notation, for $w \in \ZZ\backslash\{0\}$, let $\fs$ be the sign function where $\fs(w)=+$ if $w>0$ and $\fs(w)=-$ if $w<0$. Given $u \in \pm[m]$, set $X_{u}=X_{|u|}^{\fs(u)}$. If $\bu=(u_1,\ldots,u_k) \in (\pm [m])^k$, we define $X_\bu = X_{u_1} \cdots X_{u_k}$ and $d(\bu)=\deg X_{\bu}\in \ZZ^m$. For $\alpha=(\alpha_1,\ldots,\alpha_m)\in \ZZ^m$ we define an automorphism $\sigma^\alpha=\sigma_1^{\alpha_1}\cdots \sigma_m^{\alpha_m}$. Then $X_\bu r= \sigma^{d(\bu)}(r)X_\bu$ for all $r \in R$.
%Conjugation by the monomial $X_\bu$ determines the automorphism $\sigma^{d(\bu)}$ of $R$.
%The monomial $X_\bu$ determines an automorphism $\sigma^\bu=\sigma_{|u_1|}^{\fs(u_1)} \cdots \sigma_{|u_k|}^{\fs(u_k)}$ of $R$. 
For $\bv\in(\pm[n])^{\ell}$, we define similarly $Y_\bv$ and $d(\bv)\in \ZZ^n$. 
The $X_\bu$ (resp. $Y_\bv$) generate the TGWC $C$ (resp. $D$) over $R$ (resp. $S$). %Hence it will be sufficient to define a $\kk$-linear twisting map on the these generators. For this, we also need to define the scalars that will arise from twisting, as we do now.

In order to define a twisting map, we will first need to define certain scalars. For $1\leq j\leq m, 1\leq i\leq n$, let $q_{ij}^{(\pm,\pm)} \in \kk^\times$. For $v \in \pm [n]$ and $u \in \pm [m]$, set $q_{v,u} = q_{|v|,|u|}^{(\fs(v),\fs(u))}$. Now let $\bu \in (\pm [m])^k$ and $\bv \in (\pm [n])^\ell$. Then
\begin{align}\label{eq.qdef}
q_{(\bv,\bu)} = \prod_{v_i \in \bv,u_j \in \bu} q_{v_i,u_j}.
\end{align}

\begin{remark}\label{rem.qij}
It is not necessary that all of the $q_{ij}^{(\pm,\pm)}$ belong to the field $\kk$. It is sufficient that the $q_{ij}^{(+,+)}$ and $q_{ij}^{(-,-)}$ are central units of $R \tensor S$ that are fixed by the extensions of the $\rho_i$ and $\sigma_j$. See Theorem \ref{thm.converse}.
\end{remark}

Let $\bu,\bu'$ and $\bv,\bv'$ be sequences. We denote by $\bu\bu'$ the concatenation of $\bu$ and $\bu'$ (and similarly for $\bv,\bv'$). Thus, $X_{\bu\bu'}=X_\bu X_{\bu'}$ and so $d(\bu\bu')=d(\bu)+d(\bu')$. 
It is then clear that
\begin{align}\label{eq.pcom}
q_{(\bv,\bu)} q_{(\bv,\bu')} = q_{(\bv,\bu\bu')} \quad\text{and}\quad
q_{(\bv,\bu)} q_{(\bv',\bu)} = q_{(\bv\bv',\bu)}.
\end{align}

In the setting of Hypothesis \ref{hyp.tgwa}, we define
a $\kk$-linear map $\tau:D\otimes C\to C\otimes D$ for all $r \in R$, $s \in S$, $\bu \in (\pm [m])^k$, and $\bv \in (\pm [n])^\ell$,  as
\begin{align}\label{eq.TGWCtau}
\tau(s Y_\bv \tensor r X_\bu)
	= q_{(\bv,\bu)} \rho^{d(\bv)}(r) X_\bu \tensor \sigma^{-d(\bu)}(s) Y_\bv.
\end{align}
Equation \eqref{eq.TGWCtau} implies that $\tau(s \tensor r) = r \tensor s$ for all $s \in S$ and $r \in R$. Moreover, the map respects the $\ZZ^{m+n}$-grading on $C \tensor D$ (and $D \tensor C$).

\begin{lemma}\label{lem.tauprops}
%The map $\tau:D\otimes C\to C\otimes D$ 
The map in \eqref{eq.TGWCtau}
is well defined if and only if for all $i,j$ we have 
\begin{align}\label{eq.rhosigma-2}
\rho_i(t_j) = q_{ij}^{(+,-)}q_{ij}^{(+,+)} t_j \quad\text{and}\quad
\sigma_j(h_i) = q_{ij}^{(-,-)}q_{ij}^{(+,-)} h_i.
\end{align}
If that is the case, then
\begin{align}\label{eq.prod1-2}
q_{ij}^{(-,-)}q_{ij}^{(-,+)}q_{ij}^{(+,-)}q_{ij}^{(+,+)}=1.
\end{align}
\end{lemma}
\begin{proof}
Assume \eqref{eq.rhosigma-2}.
%Suppose that \eqref{eq.rhosigma-2} holds. 
If $X_\bu=\mu X_{\bu'}$ with $\mu\in\kk^\times$ and $\bu'$ containing the same indices as $\bu$ but potentially in a different order, then $d(\bu)=d(\bu')$ and $q_{(\bv,\bu)}=q_{(\bv,\bu')}$, so 
\begin{align*}
\tau(s Y_\bv \tensor r\mu X_{\bu'}) 
    &= q_{(\bv,\bu')} \rho^{d(\bv)}(r\mu) X_{\bu'} \tensor \sigma^{-d(\bu')}(s) Y_\bv \\
    &= q_{(\bv,\bu)} \rho^{d(\bv)}(r)\mu X_{\bu'} \tensor \sigma^{-d(\bu)}(s) Y_\bv \\
    &=q_{(\bv,\bu)} \rho^{d(\bv)}(r) X_\bu \tensor \sigma^{-d(\bu)}(s) Y_\bv
    =\tau(s Y_\bv \tensor r X_\bu).
\end{align*}
If $X_\bu=\mu\sigma^\alpha(t_i) X_{\bu'}$ with $\mu\in\kk^\times$, $\alpha\in\ZZ^m$, and $\bu'$ containing the same indices as $\bu$ minus $\{\pm i\}$ (and potentially in a different order), then $d(\bu)=d(\bu')$ and $q_{(\bv,\bu)}=q_{(\bv,(-i,i))}q_{(\bv,\bu')}$, so
\begin{align*}
\tau(s Y_\bv \tensor r\mu\sigma^\alpha(t_i) X_{\bu'}) &= q_{(\bv,\bu')} \rho^{d(\bv)}(r\mu\sigma^\alpha(t_i)) X_{\bu'} \tensor \sigma^{-d(\bu')}(s) Y_\bv \\
    &= q_{(\bv,\bu)}q_{(\bv,(-i,i))}\inv \rho^{d(\bv)}(r)\sigma^\alpha(\rho^{d(\bv)}(t_i))\mu X_{\bu'} \tensor \sigma^{-d(\bu)}(s) Y_\bv \\
    &=q_{(\bv,\bu)}q_{(\bv,(-i,i))}\inv q_{(\bv,(-i,i))} \rho^{d(\bv)}(r)\mu\sigma^{\alpha}(t_i) X_{\bu'} \tensor \sigma^{-d(\bu)}(s) Y_\bv \\
    &=q_{(\bv,\bu)}\rho^{d(\bv)}(r)X_\bu \tensor \sigma^{-d(\bu)}(s) Y_\bv
    =\tau(s Y_\bv \tensor r X_\bu),
\end{align*}
where the third equality follows from \eqref{eq.rhosigma-2}. This shows that $\tau$ preserves the relations in $C$. An analogous argument using $Y_{\bv}$ shows that it also preserves the relations in $D$, so it is well defined.

Now suppose that $\tau$ is well defined on $D\otimes C$, then we have
\begin{align*}
\rho_i(t_j) \tensor Y_i^+
    &= \tau(Y_i^+ \tensor t_j)
    = \tau(Y_i^+ \tensor X_j^-X_j^+) \\
    &= q_{ij}^{(+,-)}q_{ij}^{(+,+)} X_j^-X_j^+ \tensor Y_i^+
    = q_{ij}^{(+,-)}q_{ij}^{(+,+)} t_j \tensor Y_i^+.
\end{align*}
Similarly, $\rho_i\inv(t_j) \tensor Y_i^-=q_{ij}^{(-,-)}q_{ij}^{(-,+)} t_j \tensor Y_i^-$. This implies \eqref{eq.rhosigma-2} for the $\rho_i$ and further, $t_j=\rho_i\inv(\rho_i(t_j))$ implies \eqref{eq.prod1-2}. One obtains \eqref{eq.rhosigma-2} for the $\sigma_j$ similarly.
\end{proof}

We are now ready to show that $\tau$ defines a twisting map.

\begin{lemma}\label{lem.assoc}
Assume Hypothesis \ref{hyp.tgwa} and \eqref{eq.rhosigma-2}. The map $\tau$ defined in \eqref{eq.TGWCtau} satisfies the associativity conditions of Section \ref{sec.twistedtensor}.
\end{lemma}
\begin{proof}
Let $\bu''=\bu\bu'$ and $\bv''=\bv\bv'$. Using \eqref{eq.pcom} repeatedly, we have,
\begin{align*}
(&\mu_C \tensor \mu_D)(\id_C \tensor \tau \tensor \id_D)(\tau \tensor \tau)(\id_D \tensor \tau \tensor \id_C)(s Y_{\bv} \tensor s' Y_{\bv'} \tensor r X_{\bu} \tensor r' X_{\bu'}) \\
	&= q_{(\bv',\bu)} (\mu_C \tensor \mu_D)(\id_C \tensor \tau \tensor \id_D)(\tau \tensor \tau)(s Y_{\bv} \tensor \rho^{d(\bv')}(r) X_{\bu} \\
        &\hspace{7.3cm} \tensor \sigma^{-d(\bu)}(s') Y_{\bv'}  \tensor r' X_{\bu'}) \\
	&= q_{(\bv'',\bu)} q_{(\bv',\bu')} (\mu_C \tensor \mu_D)(\id_C \tensor \tau \tensor \id_D)(\rho^{d(\bv'')}(r) X_{\bu} \tensor \sigma^{-d(\bu)}(s) Y_{\bv} \\
        &\hspace{7.3cm} \tensor  \rho^{d(\bv')}(r') X_{\bu'} \tensor \sigma^{-d(\bu'')}(s') Y_{\bv'}  )  \\
%	&= q_{(\bv'',\bu'')}(\mu_C \tensor \mu_D)(\rho^{d(\bv'')}(r) X_{\bu} \tensor \rho^{d(\bv'')}(r') X_{\bu'}  \tensor \sigma^{-d(\bu'')}(s) Y_{\bv} \tensor \sigma^{-d(\bu'')}(s') Y_{\bv'}) \\
	&= q_{(\bv'',\bu'')}(\rho^{d(\bv'')}(r)\sigma^{d(\bu)}(\rho^{d(\bv'')}(r') ) X_{\bu''} \tensor \sigma^{-d(\bu'')}(s) \rho^{d(\bv)}(\sigma^{-d(\bu'')}(s')) Y_{\bv''}) \\
	&\stackrel{\star}{=} q_{(\bv'',\bu'')} (\rho^{d(\bv'')}(r\sigma^{d(\bu)}(r'))X_{\bu''} \tensor \sigma^{-d(\bu'')}(s\rho^{d(\bv)}(s')) Y_{\bv''} ) \\
	&= \tau( s\rho^{d(\bv)}(s') Y_{\bv''} \tensor r\sigma^{d(\bu)}(r') X_{\bu''}) \\
	&= \tau \circ (\mu_D \tensor \mu_C)(s Y_{\bv} \tensor s' Y_{\bv'} \tensor r X_{\bu} \tensor r' X_{\bu'}),
\end{align*}
where the step marked $\star$ follows because the $\sigma_i$ and $\rho_j$ are assumed to commute.
\end{proof}

\begin{lemma}\label{lem.aut_extend}
Assume Hypothesis \ref{hyp.tgwa} and \eqref{eq.rhosigma-2}. For each $i$, we extend the automorphism $\rho_i\in \Aut_\kk(R)$ to a multiplicative map $\rho_i:C\to C$ by defining $\rho_i(X_j^{\pm})=q_{ij}^{(+,\pm)} X_j^\pm$. Then $\rho_i\in \Aut_{\kk}(C)$ is a graded algebra automorphism. Similarly, for each $j$, $\sigma_j \in \Aut_\kk(S)$ extends to a graded automorphism in $\Aut_\kk(D)$.
\end{lemma}
\begin{proof}
Let $j,k \in [m]$ and $r \in R$.
That $\rho_i$ preserves the relations $X_j^+X_k^- - \mu_{jk}X_k^-X_j^+$ is clear. Because $\rho_i$ and $\sigma_j$ commute, $\rho_i$ preserves $X_j^{\pm}r - \sigma_j^{\pm}(r)X_j^{\pm}$. 
By \eqref{eq.rhosigma-2}, 
\[ \rho_i(X_j^-)\rho_i(X_j^+) - \rho_i(t_j)
    = q_{ij}^{(+,-)}q_{ij}^{(+,+)} (X_j^-X_j^+ - t_j)
    = 0.
\]
That $\rho_i$ preserves $X_j^+X_j^--\sigma_j(t_j)$ is similar.
\end{proof}

By Lemma \ref{lem.aut_extend}, we may rewrite \eqref{eq.TGWCtau} more simply as
\begin{align}
\tau(sY_\bv \tensor rX_\bu) = \rho^{d(\bv)}(rX_\bu) \tensor \sigma^{-d(\bu)}(sY_\bv).
\end{align}

\begin{proposition}\label{prop.descend}
Assume Hypothesis \ref{hyp.tgwa} and \eqref{eq.rhosigma-2}. The map $\tau$ in \eqref{eq.TGWCtau} satisfies $\tau(D\tensor\cJ+ \cI\tensor C)\subset\cJ\otimes D + C\tensor\cI$. Hence $\tau:D\tensor C\to C\otimes D$ descends to the quotient to a twisting map $B\tensor A\to A\tensor B$.
\end{proposition}
\begin{proof}
By Lemma \ref{lem.aut_extend}, we may view each $\rho_i$ as a graded automorphism of $C$. Since $\cJ$ is the sum of all graded ideals having trivial intersection with $R$, then it follows that $\rho_i(\cJ) \subset \cJ$ for each $i \in [n]$. Similarly, $\sigma_j(\cI) \subset \cI$ for each $j \in [m]$. The result now follows from \eqref{eq.TGWCtau}.
\end{proof}

By an abuse of notation, the twisting map $B \tensor A \to A \tensor B$ implied by Proposition \ref{prop.descend} will also be denoted $\tau$.
%we will also use $\tau$ to refer to the twisting map on $A \tensor B$ 
%implied by Proposition \ref{prop.descend}. 
It is then clear that the properties expressed in Lemma \ref{lem.tauprops} and Lemma \ref{lem.assoc} apply equally to this twisting map.
%the twisting map $B\tensor A\to A\tensor B$.
%Our proof that twisted tensor products of TGWAs are again TGWAs closely follows \cite[Theorem 2.16]{GR}. First we set up some notation.

%We are now ready to prove that twisted tensor products of TGWAs are (again) TGWAs. The proof closely follows \cite[Theorem 2.16]{GR}.

%\begin{hypothesis}\label{hyp.twtensor}
%\daniele{(Should this be a definition instead?)}
%Assume Hypothesis \ref{hyp.tgwa} and \eqref{eq.rhosigma-2}. Set $T=R \tensor S$. Let $\pi=(\pi_1,\hdots,\pi_{m+n})$ and $w=(w_1,\hdots,w_{m+n})$ where
%\begin{gather}
%\label{eq.piw}
%\pi_i = \begin{cases}
%\sigma_i & \text{if } i\leq m \\
%\rho_{i-m}& \text{if } i>m,
%\end{cases}\qquad
%w_i = \begin{cases}
%t_i \tensor 1 & \text{if } i\leq m \\
%1 \tensor h_{i-m}& \text{if } i>m,
%\end{cases} \\
%\label{eq.taudef}
%\eta_{ij} = \begin{cases}
%\mu_{ij} & \text{if $i,j \leq m$} \\
%\left(q_{j-m,i}^{(-,+)}\right)\inv & \text{if $i \leq m$ and $j > m$} \\
%q_{i-m,j}^{(+,-)} & \text{if $i > m$ and $j \leq m$} \\
%\nu_{ij} & \text{if $i,j > m$.}
%\end{cases}
%\end{gather}
%\end{hypothesis}

%We will first show that $(T,\pi,w)$ is a TGWD. The second part of the argument is to show that $A \tensor_\tau B$ is (isomorphic to) its corresponding TGWA.

\begin{lemma}\label{lem.twtensor}
%Assume Hypothesis \ref{hyp.twtensor}. 
Assume Hypothesis \ref{hyp.tgwa} and \eqref{eq.rhosigma-2}. Set $T=R \tensor S$. Let $\pi=(\pi_1,\hdots,\pi_{m+n})$ and $w=(w_1,\hdots,w_{m+n})$ where
\begin{gather}
\label{eq.piw}
\pi_i = \begin{cases}
\sigma_i & \text{if } i\leq m \\
\rho_{i-m}& \text{if } i>m,
\end{cases}\qquad
w_i = \begin{cases}
t_i \tensor 1 & \text{if } i\leq m \\
1 \tensor h_{i-m}& \text{if } i>m,
\end{cases} \\
\label{eq.taudef}
\eta_{ij} = \begin{cases}
\mu_{ij} & \text{if $i,j \leq m$} \\
\left(q_{j-m,i}^{(-,+)}\right)\inv & \text{if $i \leq m$ and $j > m$} \\
q_{i-m,j}^{(+,-)} & \text{if $i > m$ and $j \leq m$} \\
\nu_{ij} & \text{if $i,j > m$.}
\end{cases}
\end{gather}
Then $(T,\pi,w)$ is a regular, $\eta$-consistent TGWD of rank $m+n$.
\end{lemma}
\begin{proof}
By hypothesis, the $\pi_i$ are commuting automorphisms of $T$, and the $w_i$ are nonzero and central in $T$. Thus, $(T,\pi,w)$ is a TGWD. %, which is regular by 
Lemma \ref{lem.reg} implies regularity. 

Since $A=A_\mu(R,\sigma,t)$ (resp. $B=B_{\nu}(S,\rho,h)$) is $\mu$-consistent (resp. $\nu$-consistent), then clearly $(T,\pi,w)$ satisfies \eqref{eq.cons1} with respect to $\eta$ for $i,j \leq m$ (resp. $i,j > m$). If $i > m$ and $j \leq m$, then by \eqref{eq.rhosigma-2},
\begin{align*}
\pi_i \pi_j (w_i w_j)
    &= \rho_{i-m} \sigma_j(t_j \tensor h_{i-m})
    = \sigma_j\rho_{i-m}(t_j) \tensor \rho_{i-m}\sigma_j(h_{i-m}) \\
    &=\sigma_j(q_{i-m,j}^{(+,-)}q_{i-m,j}^{(+,+)} t_j) \tensor \rho_{i-m}(q_{i-m,j}^{(-,-)}q_{i-m,j}^{(+,-)} h_{i-m}) \\
    &=\sigma_j(q_{i-m,j}^{(+,-)}q_{i-m,j}^{(+,+)}) \rho_{i-m}(q_{i-m,j}^{(-,-)}q_{i-m,j}^{(+,-)}) \pi_i(w_i)\pi_j(w_j) \\
    &= \left(q_{i-m,j}^{(-,+)}\right)\inv q_{i-m,j}^{(+,-)} \pi_i(w_i)\pi_j(w_j)
    = \eta_{ji}\eta_{ij} \pi_i(w_i)\pi_j(w_j).
\end{align*}
A similar argument shows that $\pi_i \pi_j (w_i w_j)=\eta_{ji}\eta_{ij} \pi_i(w_i)\pi_j(w_j)$ when $i \leq m$ and $j>m$. Thus, $(T,\pi,w)$ satisfies \eqref{eq.cons1} with respect to $\eta$. 

Now we check \eqref{eq.cons2}. Again, if $i,j,k \leq m$ (resp. $i,j,k > m$), then this follows from the consistency of $A$ (resp. $B$). Suppose $i,k \leq m$ and $j> m$, then we have
%\begin{align*}
%w_j\pi_i\pi_k(w_j)
%&= h_{j-m} \sigma_i\sigma_k(h_{j-m}) \\
%&= \left( q_{j-m,i}^{(-,-)}q_{j-m,i}^{(-,+)}h_{j-m}\right) \left(q_{j-m,k}^{(-,-)}q_{j-m,k}^{(-,+)}h_{j-m}\right) \\
%&= \pi_i(w_j)\pi_k(w_j).
%\end{align*}
\begin{align*}
w_j\pi_i\pi_k(w_j)
&= h_{j-m} \sigma_i\sigma_k(h_{j-m})
= \left( q_{j-m,i}^{(-,-)}q_{j-m,i}^{(-,+)}h_{j-m}\right) \left(q_{j-m,k}^{(-,-)}q_{j-m,k}^{(-,+)}h_{j-m}\right) \\
&=\sigma_i(h_{j-m})\sigma_k(h_{j-m}) = \pi_i(w_j)\pi_k(w_j).
\end{align*}
The other checks are similar. Hence, since $(T,\pi,w)$ is regular, it is $\eta$-consistent.
\end{proof}

%Similarly, if $i \leq m$ and $j>m$, 
%\begin{align*}
%\pi_i \pi_j (w_i w_j)
%    &= \sigma_i \rho_{j-m}( t_i \tensor h_{j-m}) \\
%    &= \sigma_i \rho_{j-m}(t_i) \tensor \rho_{j-m} \sigma_i(h_{j-m}) \\
%    &= \sigma_i(q_{j-m,i}^{(+,-)}q_{j-m,i}^{(+,+)} t_i) \tensor \rho_{j-m}(q_{j-m,i}^{(-,-)}q_{j-m,i}^{(+,-)} h_{j-m}) \\ %\pi_i(\omega_i) \pi_j(\omega_j) \\
%    &= \sigma_i(q_{j-m,i}^{(+,-)}q_{j-m,i}^{(+,+)}) \rho_{j-m}(q_{j-m,i}^{(-,-)}q_{j-m,i}^{(+,-)}) \pi_i(\omega_i)\pi_j(\omega_j) \\
%    &= q_{j-m,i}^{(+,-)}\left(q_{j-m,i}^{(-,+)}\right)\inv \pi_i(\omega_i)\pi_j(\omega_j) \\
%    &= \eta_{ji}\eta_{ij} \pi_i(\omega_i)\pi_j(\omega_j).
%\end{align*}

%\begin{multline*}
%    \omega_j\pi_i\pi_k(\omega_j)
%= h_{j-m} \sigma_i\sigma_k(h_{j-m})
%=\\ \left( q_{j-m,i}^{(-,-)}q_{j-m,i}^{(-,+)}h_{j-m}\right) \left(q_{j-m,k}^{(-,-)}q_{j-m,k}^{(-,+)}h_{j-m}\right)
%= \pi_i(\omega_j)\pi_k(\omega_j).
%\end{multline*}

\begin{theorem}\label{thm.twtensor}
%Assume Hypothesis \ref{hyp.twtensor}. 
Assume Hypothesis \ref{hyp.tgwa} and \eqref{eq.rhosigma-2}, and keep the notation of Lemma \ref{lem.twtensor}.
Then $A \tensor_\tau B$ is a regular $\eta$-consistent TGWA of rank $m+n$.
\end{theorem}
\begin{proof}
By Lemma \ref{lem.twtensor}, $(T,\pi,w)$ is a regular, $\eta$-consistent TGWD of rank $m+n$. 
%Let $C$ (resp. $D$) be the TGWC associated to $(R,\sigma,t)$ (resp. $(S,\rho,h)$) and $\mu$ (resp. $\nu$) with canonical ideal $\cJ_C$ (resp. $\cJ_D$). 
Let $\widehat{C}=\cC_\eta(T,\pi,w)$ be the TGWC associated to $(T,\pi,w)$ and $\eta$ with canonical ideal $\widehat{\cJ}$ and canonical generators $\{Z_i^{\pm}\}$.

Since $(T,\pi,w)$ is $\eta$-consistent, the degree zero component of $\widehat{C}$ is $T$. 
%We can then 
Define $\psi:\widehat{C} \to A \tensor_\tau B$ to be the additive and multiplicative map that is the identity on $T$ and maps 
\[
Z_i^{\pm} \mapsto \begin{cases}
X_i^{\pm} \tensor 1 & \text{if } i\leq m \\
1 \tensor Y_{i-m}^{\pm} & \text{if } i>m.
\end{cases}\]
The map $\psi$ is clearly surjective. To verify it is well defined, %a homomorphism
we check that $\psi$ respects the defining relations on $T$. The only nontrivial check is given below. Let $i \leq m$ and $j > m$. We have
%\begin{align*}
%\psi(Z_j^-)\psi(Z_i^+)
%    &= (1\otimes Y_{j-m}^-)\cdot_\tau(X_i^+\otimes 1) \\
%    &= (\mu_A\otimes \mu_B)(1\otimes \tau(Y_{j-m}^-\otimes X_i^+)\otimes 1) \\
%    &= q_{j-m,i}^{(-,+)}(X_i^+\otimes Y_{j-m}^-) \\
%    &= q_{j-m,i}^{(-,+)}(X_i^+ \tensor 1) \cdot_\tau (1 \tensor Y_{j-m}^-) \\
%    &= q_{j-m,i}^{(-,+)} \psi(Z_i^+)\psi(Z_j^-).
%\end{align*}
\begin{align*}
\psi(Z_j^-Z_i^+)
    &= (1\otimes Y_{j-m}^-)\cdot_\tau(X_i^+\otimes 1)
    = (\mu_A\otimes \mu_B)(1\otimes \tau(Y_{j-m}^-\otimes X_i^+)\otimes 1) \\
    &= q_{j-m,i}^{(-,+)}(X_i^+\otimes Y_{j-m}^-) 
     = q_{j-m,i}^{(-,+)}(X_i^+ \tensor 1) \cdot_\tau (1 \tensor Y_{j-m}^-)
    = \eta_{ij}\inv \psi(Z_i^+ Z_j^-).
\end{align*}
%This is satisfied by our choice of $\eta_{ij}=\left(q_{j-m,i}^{(-,+)}\right)\inv$ when $i \leq m$ and $j > m$. 
A similar computation follows by the choice of $\eta_{ij}=q_{i-m,j}^{(+,-)}$ when $i > m$ and $j \leq m$.

Since $\psi$ is a $\ZZ^{m+n}$-graded map, then $\ker\psi$ is a graded ideal. Moreover, $\psi$ restricts to the identity on $T$, hence $(\ker\psi)_\bzero = \{0\}$ and $\ker\psi \subset \widehat{\cJ}$.

For the other direction, let $c \in \widehat{\cJ}$ be homogeneous of degree $g=(g_1,g_2)\in\ZZ^m\times\ZZ^n=\ZZ^{m+n}$ and write $\psi(c) = \sum_{i=1}^n a_i \tensor b_i$ such that the $\{a_i\}_i$ (and the $\{b_i\}_i$) are $\kk$-linearly independent. Because $\psi$ respects the grading on $\widehat{C}$ and $A \tensor_\tau B$, we may assume that the $a_i$ have degree $g_1$ while the $b_i$ have degree $g_2$. Choose $X_{\bu}\in A$ such that $d(\bu)=-g_1$, and $Y_\bv\in B$ such that $d(\bv)=-g_2$. Thus, $(X_\bu \tensor Y_\bv)\psi(c) \in T$. Taking a preimage $c^*\in\psi\inv(X_\bu \tensor Y_\bv)$ gives $c^* c \in T\cap\widehat{\cJ}=\{0\}$ and hence 
\begin{align*}
0&=\psi(c^*c)=(X_\bu \tensor Y_\bv)\cdot_\tau\psi(c) \\ 
&=(X_\bu \tensor Y_\bv)\cdot_\tau\sum_{i=1}^n a_i \tensor b_i = \sum_i X_\bu\rho^{d(\bv)}(a_i)\otimes \sigma^{-g_1}(Y_\bv)b_i.
\end{align*}
So, there is a $\kk$-relation amongst the $\{X_\bu\rho^{d(\bv)}(a_i)\}_i$ or the $\{\sigma^{-g_1}(Y_\bv)b_i\}_i$. Suppose we are in the first case and the second case is similar. Then for some $k_i \in \kk$, we have
$0=\sum_i k_iX_\bu\rho^{d(\bv)}(a_i)= X_\bu\rho^{d(\bv)}\left(\sum_i k_ia_i\right).$
Since $A$ is a regular TGWA, \cite[Theorem 4.3 (iv)]{HO} implies $0=\rho^{d(\bv)}\left(\sum_i k_ia_i\right)$, and so $0=\sum_i k_ia_i$, which, since $\{a_i\}_i$ were assumed to be linearly independent, means that $\psi(c)=0$.
%If $\psi(c)\neq 0$, then there is a $\kk$-relation amongst the $\{a^*a_i\}$ or the $\{b^*b_i\}$. Suppose we are in the first case and the second case is similar. Then $a^* (k_1 a_1 + \cdots + k_n a_n) = 0$ for some $k_i \in \kk$.
%there are $k_i \in \kk$ such that
%\[ a^* (k_1 a_1 + \cdots + k_n a_n) = 0.\]
%Since $A$ is a regular TGWA, \cite[Theorem 4.3 (iv)]{HO} implies $k_1 a_1 + \cdots + k_n a_n = 0$, a contradiction. 
Thus, $\ker\psi=\widehat{\cJ}$ and the induced map $\widehat{C}/\widehat{\cJ} \to A \tensor_\tau B$ is an isomorphism.
\end{proof}

The following now generalizes \cite[Corollary 2.17]{GR} and the proof is identical.

\begin{corollary}\label{cor.kfin}
Assume Hypothesis \ref{hyp.tgwa} and suppose further that $A$ and $B$ are $\kk$-finitistic. Let $\tau$ be the twisting map from \eqref{eq.TGWCtau} and let $\eta$ be as above. Then $A \tensor_\tau B$ is a regular $\eta$-consistent TGWA of rank $m+n$ which is again $\kk$-finitistic with Cartan matrix $C_{A\otimes_{\tau} B} = \mathrm{diag}(C_A,C_B)$.
%\[C_{A\otimes_{\tau} B}=\begin{bmatrix}C_A & 0 \\ 0 & C_B \end{bmatrix}.\]
\end{corollary}
%\begin{proof}
%Suppose $A$ and $B$ are $\kk$-finitistic and keep the notation \eqref{eq.piw}. If $i \leq m$ and $j > m$, then by Lemma \ref{lem.rhosig}, $\pi_i(w_j)$ is a scalar multiple of $w_j$. Consequently, $\dim_{\kk} V_{i,j} = 1$. Thus, $\dim_{\kk} V_{i,j} < \infty$ for all $i,j$. The converse is similar. 
%\end{proof}

%We close this section by giving a partial converse to Theorem \ref{thm.twtensor}. 
%In particular, we show 
The following is a partial converse to Theorem \ref{thm.twtensor}.
In particular, this shows that the twisting maps defined in \eqref{eq.TGWCtau} are in some sense the \emph{only} graded twisting maps between two TGWAs in which the twisted tensor product is again a TGWA.

\begin{theorem}\label{thm.converse}
Let $A$ and $B$ be TGWAs satisfying Hypothesis \ref{hyp.tgwa}. Suppose that $\tau:B \tensor A \to A \tensor B$ is a graded twisting map such that $A \tensor_\tau B$ is a regular TGWA generated over $R \tensor S$ by the $X_j^{\pm} \tensor 1$ and $1 \tensor Y_i^{\pm}$.
Then the following hold:
\begin{enumerate}
    \item the $\sigma_j$'s and $\rho_i$'s extend to commuting automorphisms of $R \tensor S$, and
    \item $\tau$ satisfies \eqref{eq.TGWCtau} where the $q_{ij}^{(\pm,\pm)}$ are central units of $R \tensor S$ which are fixed by (the extensions of) $\sigma_j$ and $\rho_i$ for all $i,j$.
\end{enumerate}
\end{theorem}
\begin{proof}
(1) Let $r \tensor s \in R \tensor S$ and $j \in [n]$. Then $(1 \tensor Y_i)(r \tensor s) = \varrho_i(r \tensor s)(1 \tensor Y_i)$ for some automorphisms $\varrho_i$ of $R \tensor S$. By taking $r=1_R$, we see that $\varrho_i$ restricts to $\rho_i$ on $S$. A similar argument holds for $\sigma_j$. Then the extensions of $\sigma_j$ and $\rho_i$ commute by definition of a TGWA.

(2) As $\tau$ is a well-defined graded twisting map, then necessarily $\tau$ satisfies \eqref{eq.TGWCtau} for some $q_{ij}^{(\pm,\pm)} \in R \tensor S$, and we have \eqref{eq.rhosigma-2} (applied to $t_j\tensor 1$ and $1\tensor h_i$) and \eqref{eq.prod1-2}. By definition of a TGWA, $(1 \tensor Y_i^+)(X_j^- \tensor 1) = \eta (X_j^- \tensor 1)(1 \tensor Y_i^+)$ for some $\eta \in \kk^\times$ and, by regularity, $\eta = q_{ij}^{(+,-)}$. A similar argument shows that $q_{ij}^{(-,+)} \in \kk^\times$. Thus, all claims in (2) trivially apply to $q_{ij}^{(+,-)}$ and $q_{ij}^{(-,+)}$.

By \eqref{eq.rhosigma-2}, since $\rho_i$ is invertible, then $q_{ij}^{(+,-)}q_{ij}^{(+,+)}$ is a unit, which in turn implies that $q_{ij}^{(+,+)}$ is a unit since $q_{ij}^{(+,-)} \in \kk^\times$. Furthermore, $t_j \tensor 1$ is central in $R \tensor S$ and automorphisms preserve the center, so $q_{ij}^{(+,-)}q_{ij}^{(+,+)}$ is central. Thus, $q_{ij}^{(+,+)}$ is central.
Now we observe that for any $i,k \in [n]$, $i \neq k$, and $j \in [m]$,
%\begin{align*}
%\nu_{ij}\rho_i(q_{jk}^{(-,+)})q_{ik}^{(+,+)}(X_k^+ \tensor Y_j^-Y_i^+) 
%    &= \rho_i\left(q_{jk}^{(-,+)}\right)q_{ik}^{(+,+)}(X_k^+ \tensor Y_i^+Y_j^-) \\ 
%    &= (1 \tensor Y_i^+)(1 \tensor Y_j^-)(X_k^+ \tensor 1) \\
%    &= \nu_{ij}(1 \tensor Y_j^-)(1 \tensor Y_i^+)(X_k^+ \tensor 1) \\
%    &= \nu_{ij}\rho_j(q_{ik}^{(+,+)})q_{jk}^{(-,+)}(X_k^+ \tensor Y_j^-Y_i^+).
%\end{align*}
%\begin{align*}
%\nu_{ij}\rho_i&(q_{jk}^{(-,+)})q_{ik}^{(+,+)}(X_k^+ \tensor Y_j^-Y_i^+) \\
%    &= \rho_i\left(q_{jk}^{(-,+)}\right)q_{ik}^{(+,+)}(X_k^+ \tensor Y_i^+Y_j^-)
%    = (1 \tensor Y_i^+)(1 \tensor Y_j^-)(X_k^+ \tensor 1) \\
%    &= \nu_{ij}(1 \tensor Y_j^-)(1 \tensor Y_i^+)(X_k^+ \tensor 1)
%    = \nu_{ij}\rho_j(q_{ik}^{(+,+)})q_{jk}^{(-,+)}(X_k^+ \tensor Y_j^-Y_i^+).
%\end{align*}
\begin{align*}
\nu_{ik}\rho_i&(q_{kj}^{(-,+)})q_{ij}^{(+,+)}(X_j^+ \tensor Y_k^-Y_i^+) \\
    &= \rho_i\left(q_{kj}^{(-,+)}\right)q_{ij}^{(+,+)}(X_j^+ \tensor Y_i^+Y_k^-)
    = (1 \tensor Y_i^+)(1 \tensor Y_k^-)(X_j^+ \tensor 1) \\
    &= \nu_{ik}(1 \tensor Y_k^-)(1 \tensor Y_i^+)(X_j^+ \tensor 1)
    = \nu_{ik}\rho_k(q_{ij}^{(+,+)})q_{kj}^{(-,+)}(X_j^+ \tensor Y_k^-Y_i^+).
\end{align*}
Since $q_{kj}^{(-,+)} \in \kk^\times$, then by comparing coefficients we have $\rho_k(q_{ij}^{(+,+)})=q_{ij}^{(+,+)}$ for $i\neq k$. 
For $j\in [m]$, $i\in[n]$ we also have
\begin{align*}
q_{ij}^{(-,+)}&q_{ij}^{(+,+)}(X_j^+ \tensor \rho_i(h_i)) 
    =q_{ij}^{(-,+)}q_{ij}^{(+,+)}(X_j^+ \tensor Y_i^+Y_i^-) \\
    &=(1 \tensor Y_i^+) q_{ij}^{(-,+)}(X_j^+ \tensor 1)(1\tensor Y_i^-) 
    = (1 \tensor Y_i^+)(1 \tensor Y_i^-)(X_j^+ \tensor 1) \\
    &= (1\otimes \rho_i(h_i))(X_j^+ \tensor 1) 
    =  (X_j^+ \tensor 1)(1\tensor\sigma_j\inv\rho_i(h_i))\\
    &= (X_j^+\tensor 1)q_{ij}^{(-,+)}q_{ij}^{(+,+)}(1\tensor \rho_i(h_i))
    = q_{ij}^{(-,+)}\sigma_j(q_{ij}^{(+,+)})(X_j^+\tensor \rho_i(h_i))
\end{align*}
which shows that $\sigma_j(q_{ij}^{(+,+)})=q_{ij}^{(+,+)}$.
A similar computation but with the order of $Y_i^+$ and $Y_i^-$ reversed
now shows that
\[ 
q_{ij}^{(-,+)}\rho_i\inv(q_{ij}^{(+,+)})(X_j^+ \tensor h_i)
    = q_{ij}^{(-,+)}\sigma_j(q_{ij}^{(+,+)})(X_j^+\tensor h_i),
\]
%, and
%\begin{align*}
%q_{ij}^{(-,+)}&\rho_i\inv(q_{ij}^{(+,+)})(X_j^+ \tensor h_i) 
%    =q_{ij}^{(-,+)}\rho_i\inv(q_{ij}^{(+,+)})(X_j^+ \tensor Y_i^-Y_i^+) \\
%    &=(1 \tensor Y_i^-) q_{ij}^{(+,+)}(X_j^+ \tensor 1)(1\tensor Y_i^+) 
%    = (1 \tensor Y_i^-)(1 \tensor Y_i^+)(X_j^+ \tensor 1) \\
%    &= (1\otimes h_i)(X_j^+ \tensor 1) 
%    =  (X_j^+ \tensor 1)(1\tensor\sigma_j\inv(h_i))\\
%    &= (X_j^+\tensor 1)q_{ij}^{(-,+)}q_{ij}^{(+,+)}(1\tensor h_i)
%    = q_{ij}^{(-,+)}\sigma_j(q_{ij}^{(+,+)})(X_j^+\tensor h_i)
%\end{align*}
which gives $\rho_i\inv(q_{ij}^{(+,+)})=\sigma_j(q_{ij}^{(+,+)})=q_{ij}^{(+,+)}$. Hence $q_{ij}^{(+,+)}$ is invariant under $\rho_k$ for all $k\in[m]$ and similar arguments show that it is also invariant under $\sigma_\ell$, $\ell\in[n]$, and analogously for $q_{ij}^{(-,-)}$, which proves that all the claims in (2) apply.
\end{proof}

%Next we see that
%\begin{align*}
%\rho_i(t_j \tensor 1)(1 \tensor Y_i^+)
%    &= (1 \tensor Y_i^+)(t_j \tensor 1)
%    = (1 \tensor Y_i^+)(X_j^- \tensor 1)(X_j^+ \tensor 1) \\
%    &= q_{ij}^{(+,-)}q_{ij}^{(+,+)}(t_j \tensor 1)(1 \tensor Y_i^+).
%\end{align*}
%Since $\rho_i$ is invertible, then $q_{ij}^{(+,-)}q_{ij}^{(+,+)}$ is a unit, which in turn implies that $q_{ij}^{(+,+)}$ is a unit since $q_{ij}^{(+,-)} \in \kk^\times$. Furthermore, $t_j \tensor 1$ is central in $R \tensor S$ and automorphisms preserve the center, so $q_{ij}^{(+,-)}q_{ij}^{(+,+)}$ is central. Thus, $q_{ij}^{(+,+)}$ is central. 
%So, the claims in (2) all apply to $q_{ij}^{(+,+)}$ and similar arguments show the same for the $q_{ij}^{(-,-)}$.

We end this section with some examples. 
%The second shows that it is not strictly necessary to extend the automorphisms $\sigma$ and $\rho$ trivially to $R \tensor S$.}`
The second shows that it is possible to extend $\sigma$ and $\rho$ to commuting automorphisms of $R\otimes S$ non-trivially.

\begin{example}\label{ex.mpwa}    
(1) Let $\mathcal{A}$ be the algebra from Example \ref{ex.ttps}(3) with $n=2$.
For each $i$, let $A_i$ be the (rank one) GWA
$\kk[z_i](\sigma_i,a_i)$ with $\sigma_i(z_i)=q_i z_i+1$ and $a_i=z_i$. Let $x_i,y_i$ be the canonical generators of $A_i$ (so $x_i=X_i^+$ and $y_i=X_i^-$). Then
\[ x_iy_i - q_i y_ix_i = \sigma(a_i) - q_i a_i = (q_i z_i+1) - q_i z_i = 1.\]
Thus, $A_i \iso A_1^{q_i}(\kk)$, the first quantum Weyl algebra, and so $A_i$ is (trivially) a TGWA of Cartan type $A_1$.
We extend $\sigma_1$ and $\sigma_2$ trivially to automorphisms of $\kk[z_1,z_2]$.
For $i \neq j$, set 
$q_{ij}^{(+,+)} = \lambda_{ij}$,
$q_{ij}^{(-,+)} = \lambda_{ji}$,
$q_{ij}^{(+,-)} = \lambda_{ji}$, and
$q_{ij}^{(-,-)} = \lambda_{ij}$.
%\[ 
%q_{ij}^{(+,+)} = \lambda_{ij}, \quad
%q_{ij}^{(-,+)} = \lambda_{ji}, \quad
%q_{ij}^{(+,-)} = \lambda_{ji}, \quad
%q_{ij}^{(-,-)} = \lambda_{ij}.\]
It is clear that these parameters satisfy \eqref{eq.prod1-2}.
Let $\tau$ be the twisting map defined in \eqref{eq.TGWCtau} (with $A=A_1$ and $B=A_2$). Then $\cA \iso A_1 \tensor_\tau A_2$. Inductively, we may obtain any algebra from Example \ref{ex.ttps}(3).

(2) Let $R=\kk[t_1,t_2]$. Define $\sigma_1,\sigma_2 \in \Aut(R)$ by
\[
\sigma_1(t_1) = t_1+t_2^2, \quad \sigma_1(t_2) = t_2, \quad
\sigma_2(t_1) = t_1, \quad \sigma_2(t_2) = -t_2.
\]
Let $A$ be the rank 2 TGWA defined by this data. Then $A$ is of Cartan type $(A_1)^2$ since the $V_{ij}$, $i \neq j$, are both 1-dimensional. It is not clear whether or not $A$ can be expressed as the twisted tensor product of two rank 1 GWAs.

Let $S=\kk[h]$ and define $\rho \in \Aut(S)$ by $\rho(h)= qh$ for some nonroot of unity $q \in \kk^\times$. This data defines a rank one GWA denoted B. 

We now construct a twisted tensor product of $A$ and $B$. Extend the automorphisms $\sigma_1$ and $\sigma_2$ to $R \tensor S$ trivially. On the other hand, extend $\rho$ to $R \tensor S$ by setting $\rho(t_1)=-t_1$ and $\rho(t_2)= i t_2$. Then $\rho$ and $\sigma$ commute. 
%In particular, the only nontrivial relation to check is,
%\[ \rho(\sigma_1(u_1)) = \rho(u_1+u_2^2) = -u_1 - u_2^2 = \sigma_1(-u_1) = \sigma_1(\rho(u_1)).\]
Thus, it is clear that $A \tensor_\tau B$ is (isomorphic to) a rank 3 TGWA of Cartan type $(A_1)^3$.
\end{example}

\section{Graded twists of TGWAs}

In this section we prove that the graded twist of a TGWA is again a TGWA. First we begin with a discussion of TGWCs.

\begin{hypothesis}\label{hyp.twist}
Keep Hypothesis \ref{hyp.tgwa}. For $\alpha \in \ZZ^m$, let $\gt_\alpha$ be a graded $\kk$-linear automorphism of $C$ such that $\gt:\ZZ^m\to \Aut_\kk(C)$ is a homomorphism and
\begin{align}\label{eq.twist}
%\gt_{\pm \be_i}(X_j^{\pm})=q_{ij}^{(\pm,\pm)} X_j^{\pm}, 
\gt_{\be_i}(X_j^{\pm})=q_{ij}^{(+,\pm)} X_j^{\pm}, \quad
\gt_{-\be_i}(X_j^{\pm})=q_{ij}^{(-,\pm)} X_j^{\pm}, 
\qquad q_{ij}^{(\pm,\pm)} \in \kk^\times, i,j \in [m].
\end{align}
\end{hypothesis}

Note that, since $\chi_{-\be_i}=(\chi_{\be_i})\inv$, then $q_{ij}^{(-,+)}q_{ij}^{(+,+)} = q_{ij}^{(-,-)}q_{ij}^{(+,-)} = 1$.

As in Remark \ref{rem.qij}, the restriction that the scalars $q_{ij}^{(\pm,\pm)}$ in \eqref{eq.twist} are in $\kk^\times$ may be relaxed to assuming that they are central units of $R$ fixed by the $\sigma_i$. (However, we are most interested in the case that the graded twists $\gt$ correspond to cocycle twists and so we keep this hypothesis.)

By Remark \ref{rem:twist-system}, Hypothesis \ref{hyp.twist} implies that $\gt=\{\gt_\alpha \mid \alpha \in \ZZ^m\}$ is a (left and right) twisting system on $C$ and further that $\gt_\bzero=\id_C$. The action $\gt_\alpha$ on an arbitrary monomial can be computed using a similar formula as \eqref{eq.qdef}.

\begin{lemma}\label{lem.twist}
Assume Hypothesis \ref{hyp.twist}. Then the following hold:
\begin{enumerate}
    \item $\gt_\alpha \sigma_i = \sigma_i \gt_\alpha$ for each $i \in [m]$ and each $\alpha \in \ZZ^m$, and
    \item $\gt_{\pm\be_i}(t_j) = q_{ij}^{(\pm,+)}q_{ij}^{(\pm,-)} t_j$.
%    \item $q_{ij}^{(-,+)}q_{ij}^{(+,+)} = q_{ij}^{(-,-)}q_{ij}^{(+,-)} = 1$.
\end{enumerate}
\end{lemma}
\begin{proof}
(1) It suffices to prove this for $\alpha=\be_j$, $j \in [m]$. Let $r \in R$, then
\[ 0 = \chi_{\be_j}(X_i^+r - \sigma_i(r)X_i^+)
%    = q_{ji}^{(+,+)} (X_i^+ \chi_{\be_j}(r) - \chi_{\be_j}(\sigma_i(r))X_i^+)
    = q_{ji}^{(+,+)} (\sigma_i(\chi_{\be_j}(r)) - \chi_{\be_j}(\sigma_i(r))X_i^+.
\]
Since $C$ is regular, then the result follows.

(2) %This follows from an identical argument as in Lemma \ref{lem.tauprops}.
For $i,j\in[m]$, $\chi_{\pm\be_i}(t_j)=\chi_{\pm\be_i}(X_j^-X_j^+)=q_{ij}^{(\pm,-)}q_{ij}^{(\pm,+)}t_j$. 
%\[ \chi_{\pm\be_i}(t_j)=\chi_{\pm\be_i}(X_j^-X_j^+)=q_{ij}^{(\pm,-)}X_j^-q_{ij}^{(\pm,+)}X_j^+=q_{ij}^{(\pm,-)}q_{ij}^{(\pm,+)}t_j. \qedhere\]
%(3) Because $\gt$ is a homomorphism and $\gt_\bzero = \id_C$, then
%\[ X_j^+ = \gt_\bzero(X_j^{+}) 
%    = \gt_{-\be_i}\gt_{\be_i}(X_j^+)
%    = q_{ij}^{(-,+)}q_{ij}^{(+,+)} X_j^+.\]
%So $q_{ij}^{(-,+)}q_{ij}^{(+,+)}=1$. The other result follows from replacing $X_j^+$ with $X_j^-$.
\end{proof}
%\[ X_j^- = \gt_\bzero(X_j^-) = \gt_{-\be_i}\gt_{\be_i}(X_j^-) = q_{ij}^{(-,-)}q_{ij}^{(+,-)} X_j^-\]

The automorphisms $\gt_\alpha$ of $C$ defined in \eqref{eq.twist} clearly descend to automorphisms of $A$. Thus, $\gt$ satisfying Hypothesis \ref{hyp.twist} and \eqref{eq.twist} may equally be regarded as a twisting system of $A$.

\begin{theorem}\label{thm.TGWAtwist}
Assume Hypothesis \ref{hyp.twist} with $\gt$ defined as in \eqref{eq.twist}. 
\begin{enumerate}
    \item If $\nu_{ij} = \mu_{ij} q_{ji}^{(-,+)}q_{ij}^{(-,-)}$
    for all $i \neq j$, then $\ltC\iso C_\nu(R,\gt\inv\sigma,t)$ is a TGWC.
    %Then $\ltC\iso C_\nu(R,\gt\inv\sigma,t)$ where 
    %$\nu_{ij} = \mu_{ij} q_{ji}^{(-,+)}q_{ij}^{(-,-)}$
    %for all $i \neq j$. In particular, $\ltC$ is a TGWC. 
    \item If $\lt{\cJ}$ denotes the canonical ideal for $\ltC$, then $\cJ=\lt{\cJ}$ as sets.
    %Let $\lt{\cJ}$ denote the canonical ideal for $\ltC$. Then, as sets, $\cJ=\lt{\cJ}$.
    \item Then $\ltA \iso \ltC/\lt{\cJ}$. So, $\ltA\iso A_\nu(R,\gt\inv\sigma,t)$ is a rank $m$ TGWA.
\end{enumerate}
\end{theorem}
\begin{proof}
(1) For generators in $\ltC$ we take $X_i^+$ and $(X_i^-)'=q_{ii}^{(-,-)}X_i^-$, $i\in [m]$. Then, for $i,j \in [m]$, $i \neq j$, we have in $\ltC$,
%\begin{align*}
%X_i^+ \star X_j^- - \nu_{ij} X_j^- \star X_i^+
%    &= \gt_{-\be_j}(X_i^+) X_j^- - \nu_{ij} \gt_{\be_i}(X_j^-)X_i^+ \\
%    &= q_{ji}^{(-,+)} X_i^+ X_j^- - \nu_{ij} q_{ij}^{(+,-)} X_j^- X_i^+ \\
%    &= q_{ji}^{(-,+)} \left( X_i^+ X_j^- - \nu_{ij} q_{ij}^{(+,-)}\left(q_{ji}^{(-,+)}\right)\inv X_j^- X_i^+\right) = 0.
%\end{align*}
\begin{align*}X_i^+ \star (X_j^-)'&=\gt_{-\be_j}(X_i^+) (X_j^-)'=q_{ji}^{(-,+)}q_{ii}^{(-,-)} X_i^+ X_j^-=q_{ji}^{(-,+)}\mu_{ij}q_{ii}^{(-,-)}  X_j^-X_i^+ \\ &=q_{ji}^{(-,+)}\mu_{ij}\gt_{\be_i}\inv ((X_j^-)')\star X_i^+= \mu_{ij} q_{ji}^{(-,+)}q_{ij}^{(-,-)}(X_j^-)'\star X_i^+. \end{align*}
%By Lemma \ref{lem.twist}, $\nu$ is as described. 
For $r \in R$, we have
\begin{align*}
X_i^+ \star r - \gt_{\be_i}\inv\sigma_i(r) \star X_i^+
    &= X_i^+ r - \sigma_i(r)X_i^+ = 0 \\
(X_i^-)' \star r - (\gt_{\be_i}\inv\sigma_i)\inv(r) \star (X_i^-)'
    &= q_{ii}^{(-,-)} X_i^- r -q_{ii}^{(-,-)} \sigma_i\inv(r)X_i^- = 0.
\end{align*}
%In $\ltC$ we replace $X_i^-$ with $(X_i^-)' = \left(q_{ii}^{(+,-)}\right)\inv X_i^-= q_{ii}^{(-,-)}X_i^-$. 
Then,
\[ (X_i^-)' \star X_i^+ - t_i
 = q_{ii}^{(-,-)} \gt_{\be_i}(X_i^-) X_i^+ - t_i
    = X_i^- X_i^+ - t_i = 0.\]
Finally, applying Lemma \ref{lem.twist} we have,
\begin{align*}
X_i^+ \star (X_i^-)' - \chi_{\be_i}\inv \sigma_i(t_i)
    &= q_{ii}^{(-,-)} \gt_{-\be_i}(X_i^+) X_i^- - q_{ii}^{(-,+)}q_{ii}^{(-,-)} \sigma_i(t_i) \\
    &= q_{ii}^{(-,+)}q_{ii}^{(-,-)} (X_i^+X_i^- - \sigma_i(t_i)) = 0.
\end{align*}
This proves (1).

(2) Let $c \in \cJ$. Then there exists an $r\in Z(R) \cap R_{\text{reg}}$ such that $rc=0$. Since $\cJ$ is a graded ideal, we can assume that $c\in C_\alpha$ is homogeneous, hence $\chi_\alpha\inv(r) \star c = rc = 0$, with $\chi_\alpha\inv(r)\in Z(R) \cap R_{\text{reg}}$. Thus, $c \in \lt{\cJ}$ so ${\cJ} \subset \lt{\cJ}$. The converse follows by noting that ${}^{\gt\inv} \! (\lt{\cJ}) = {\cJ}$.

(3) Let $\pi:C \to C/\cJ = A$ be the quotient map. Since, as sets, $C=\ltC$, $\cJ=\lt{\cJ}$, and $A=\ltA$, then $\pi:\ltC \to \ltA$ is a $\kk$-linear surjection. The fact that it is an algebra homomorphism follows directly because $\cJ=\lt{\cJ}$ is a homogeneous ideal and hence the grading (and the twisting) on $C=\lt{C}$ is the same as the one on $A=\lt{A}$, so $\lt{A}\iso \lt{C}/\ker \pi=\lt{C}/\lt{\cJ}$. %We need only verify that this is a homomorphism. 
%Note that $\cJ$ is a homogeneous ideal and therefore the grading on $A=\ltA$ is equivalent to that on $C=\ltC$. We have
%\begin{align*}
%\pi( X_\beta \star X_\alpha )
%	&= \pi(\gt_\alpha(X_\beta) X_\alpha)
%	= \gt_\alpha(X_\beta) X_\alpha + \cJ
%	= \gt_\alpha(X_\beta+\cJ)(X_\alpha + \cJ) \\
%	&= \gt_\alpha(X_\beta+\lt{\cJ}) \star (X_\alpha + \lt{\cJ})
%	= \pi(X_\beta) \star \pi(X_\alpha).\qedhere
%\end{align*}
\end{proof}

\begin{corollary}\label{cor.twist-props}
Assume Hypothesis \ref{hyp.twist} with $\gt$ defined as in \eqref{eq.twist}. 
%Then the following hold:
\begin{enumerate}
    \item If $A$ is regular, then $\ltA$ is regular.
    \item If $A$ is $\mu$-consistent, then $\ltA$ is $\nu$-consistent, with $\nu$ as in Thm \ref{thm.TGWAtwist}(1).
    \item If $A$ is $\kk$-finitistic, then $\ltA$ is $\kk$-finitistic of the same Cartan type.
\end{enumerate}
\end{corollary}
\begin{proof}
%By Theorem \ref{thm.TGWAtwist}, $\ltA$ is a TGWA. 
Since $\gt_\bzero=\id_A$ acts trivially on $R$, then (1) is clear.

For (2), suppose $A$ is $\mu$-consistent. Then by Lemma \ref{lem.twist},
\begin{align*}
\nu_{ij}&\nu_{ji}\chi_{\be_i}\inv\sigma_i(t_i)\chi_{\be_j}\inv\sigma_j(t_j) \\
    &= \mu_{ij} q_{ji}^{(-,+)}q_{ij}^{(-,-)}
        \mu_{ji} q_{ij}^{(-,+)}q_{ji}^{(-,-)}
		q_{ii}^{(-,+)}q_{ii}^{(-,-)}\sigma_i(t_i)
        q_{jj}^{(-,+)}q_{jj}^{(-,-)}\sigma_j(t_j) \\
%	&= \left(\mu_{ij} q_{ji}^{(-,+)}q_{ij}^{(-,-)}\right)\left(\mu_{ji} q_{ij}^{(-,+)}q_{ji}^{(-,-)}\right)
%		\left(q_{ii}^{(-,+)}q_{ii}^{(-,-)}\sigma_i(t_i)\right)\left(q_{jj}^{(-,+)}q_{jj}^{(-,-)}\sigma_j(t_j)\right) \\
	&= \left(q_{ji}^{(-,+)}q_{ij}^{(-,-)}\right)\left(q_{ij}^{(-,+)}q_{ji}^{(-,-)}\right)
		\left(q_{ii}^{(-,+)}q_{ii}^{(-,-)}\right)\left(q_{jj}^{(-,+)}q_{jj}^{(-,-)}\right)\sigma_i\sigma_j(t_it_j) \\
	&= (\chi_{\be_i}\inv\sigma_i)(\chi_{\be_j}\inv\sigma_j)(t_it_j)
\end{align*}
and
\begin{align*}
t_j\chi_{\be_i}\inv\sigma_i\chi_{\be_k}\inv\sigma_k(t_j) 
    &= t_j q_{ij}^{(-,+)}q_{ij}^{(-,-)}q_{kj}^{(-,+)}q_{kj}^{(-,-)}\sigma_i\sigma_k(t_j) \\
    &= q_{ij}^{(-,+)}q_{ij}^{(-,-)}q_{kj}^{(-,+)}q_{kj}^{(-,-)}\sigma_i(t_j)\sigma_k(t_j) 
    = \chi_{\be_i}\inv\sigma_i(t_j)\chi_{\be_k}\inv\sigma_k(t_j).
\end{align*}
This shows that $\ltA$ is $\nu$-consistent. 
%\daniele{(This only shows \eqref{eq.cons1}, don't we need another computation for \eqref{eq.cons2}?)
%\begin{multline*}
%t_j\chi_{\be_i}\inv\sigma_i\chi_{\be_k}\inv\sigma_k(t_j) 
%    = t_j q_{ij}^{(-,+)}q_{ij}^{(-,-)}q_{kj}^{(-,+)}q_{kj}^{(-,-)}\sigma_i\sigma_k(t_j) \\
%    = q_{ij}^{(-,+)}q_{ij}^{(-,-)}q_{kj}^{(-,+)}q_{kj}^{(-,-)}\sigma_i(t_j)\sigma_k(t_j) 
%    = \chi_{\be_i}\inv\sigma_i(t_j)\chi_{\be_k}\inv\sigma_k(t_j).
%\end{multline*}}

(3) Since $\chi_{\be_i}\inv\sigma_i(t_j)$ is a nonzero scalar multiple of $\sigma_i(t_j)$, then
\[ \Span_\kk\{\sigma_i^k(t_j) ~|~ k \in \ZZ\}=\Span_\kk\{(\chi_{\be_i}\inv\sigma_i)^k(t_j) ~|~ k \in \ZZ\}\] 
and the result follows.
%so the $\kk$-finitistic property is preserved. See \eqref{eq.vij}.
%The $\kk$-finitistic property only depends on the dimension of the $V_{ij}$. As $\chi_{\be_i}\inv\sigma_i(t_j)$ is a scalar multiple of $\sigma_i(t_j)$, then this dimension is equal in the corresponding set for $\ltA$.
\end{proof}

The next result may be seen as a generalization of a known result in the context of multiparameter quantized Weyl algebras (see \cite{FH1,GYtwist}).

\begin{corollary}\label{cor.TGWAtwist}
Let $C=C_\mu(R,\sigma,t)$, and $A=A_\mu(R,\sigma,t)$ as in Hypothesis \ref{hyp.tgwa}, and let $C'=C_\nu(R,\sigma,t)$, $A'=A_\nu(R,\sigma,t)$. Then $C'\iso\ltC$ (resp. $A'\iso\ltA$), for some $\chi$ as in Hypothesis \ref{hyp.twist}, if and only if $\nu_{ij}\nu_{ji}=\mu_{ij}\mu_{ji}$ for all $i \neq j$.
%Assume Hypothesis \ref{hyp.twist}. Let $C'=C_\nu(R,\sigma,t)$ be a TGWC of rank $m$ and let $A'=A_\nu(R,\sigma,t)$ be the corresponding TGWA. Then $C'=\ltC$ (resp. $A'=\ltA$) 
%is a graded twist of $C$ (resp. $A$) if and only if $\nu_{ij}\nu_{ji}=\mu_{ij}\mu_{ji}$ for all $i \neq j$.
\end{corollary}
\begin{proof}If $C'\iso \ltC$ for some $\chi$, by Theorem \ref{thm.TGWAtwist}(1) we have $C'=C_\nu(R,\sigma,t)\iso C_\nu(R,\gt\inv\sigma,t)=\ltC$, with 
$\nu_{ij} = \mu_{ij} q_{ji}^{(-,+)}q_{ij}^{(-,-)}$, in particular $\chi_\alpha=\id_R$ for all $\alpha$. Consequently, by Lemma \ref{lem.twist}(2),
\begin{align}\label{eq.qij-reln}
q_{ij}^{(+,+)}q_{ij}^{(+,-)}=q_{ij}^{(-,+)}q_{ij}^{(-,-)}=1 \quad\text{for all $i,j \in [m]$}.
\end{align}
Then
\[ \nu_{ij}\nu_{ji} 
	= \left(\mu_{ij} q_{ji}^{(-,+)}q_{ij}^{(-,-)}\right)\left(\mu_{ji} q_{ij}^{(-,+)}q_{ji}^{(-,-)}\right)
%	= \mu_{ij}\mu_{ji}
%	\left( q_{ij}^{(-,-)}q_{ij}^{(-,+)}\right)\left( q_{ji}^{(-,+)}q_{ji}^{(-,-)}\right)
	= \mu_{ij}\mu_{ji}.
\]
Conversely, we can define the needed $\chi$ by choosing any $q_{ij}^{(\pm,\pm)}$ satisfying \eqref{eq.qij-reln}.
%By Theorem \ref{thm.TGWAtwist}(1), $\ltC\iso C_\nu(R,\gt\inv\sigma,t)$ where $\nu_{ij} = \mu_{ij} q_{ji}^{(-,+)}q_{ij}^{(-,-)}$. But the assumption that $C'=C_\nu(R,\sigma,t)$ implies that $\chi_\alpha(r)=r$ for all $r \in R$. Consequently, by Lemma \ref{lem.twist}(2),
%\begin{align}\label{eq.qij-reln}
%q_{ij}^{(+,+)}q_{ij}^{(+,-)}=q_{ij}^{(-,+)}q_{ij}^{(-,-)}=1 \quad\text{for all $i,j \in [m]$}.
%\end{align}
%Then
%\[ \nu_{ij}\nu_{ji} 
%	= \left(\mu_{ij} q_{ji}^{(-,+)}q_{ij}^{(-,-)}\right)\left(\mu_{ji} q_{ij}^{(-,+)}q_{ji}^{(-,-)}\right)
%	= \mu_{ij}\mu_{ji}
%	\left( q_{ij}^{(-,-)}q_{ij}^{(-,+)}\right)\left( q_{ji}^{(-,+)}q_{ji}^{(-,-)}\right)
%	= \mu_{ij}\mu_{ji}.\]
%The converse is clear by choosing $q_{ij}^{(\pm,\pm)}$ satisfying \eqref{eq.qij-reln}.
\end{proof}

%\begin{corollary}\label{cor.A2}
%Let $A=A_\mu(R,\sigma,t)$ be a $\kk$-finitistic TGWA of Cartan type $A_2$ with minimal polynomials $p_{12}=p_{21}=(x-1)^2$ and $\mu_{12}\mu_{21}=1$. Then $A$ is noetherian.
%\end{corollary}
%\begin{proof}
%Follows from Theorem \ref{thm.TGWAtwist}, \cite[Theorem 5.27]{GR}, and \cite[Proposition 5.1]{ztwist}.
%\daniele{So, in \cite[Thm 5.27]{GR} we refer to \cite[Example 5.6]{GR}, is it clear that any TGWA of type $A_2$ as defined in the statement comes from those? I guess the point is that the proof of \cite[Thm 5.27]{GR} works quite generally for anything that has those minimal polynomials?} 
%\end{proof}

\begin{example}\label{ex.A2}
Let $R=\kk[h]$. Define $\sigma_1,\sigma_2 \in \Aut(R)$ by $\sigma_1(h)=h+1$ and $\sigma_2(h)=h-1$. Set $t_1=h$ and $t_2=h+1$, and set $\mu_{12}=\mu_{21}=1$. Then $A=A_\mu(R,\sigma,t)$ is a $\kk$-finitistic TGWA of Cartan type $A_2$ with minimal polynomials $p_{12}=p_{21}=(x-1)^2$. By \cite[Theorem 5.27]{GR}, $A$ is noetherian. From \cite[Proposition 5.1]{ztwist}, $\ltA$ is also noetherian. By Theorem \ref{thm.TGWAtwist}, it follows that $A_\nu(R,\sigma,t)$
is noetherian provided $\nu_{12}\nu_{21}=1$.
\end{example}

\subsection{Twisted tensor products and graded twists}

Assume Hypothesis \ref{hyp.tgwa}, where the automorphisms $\sigma$ and $\rho$ are extended trivially to the other factor in $R\otimes S$. Since $\rho_i(t_j)=t_j$ and $\sigma_j(h_i)=h_i$ for all $i,j$, 
\eqref{eq.rhosigma-2} implies that 
\begin{equation}\label{eq:qij-inv}
 q_{ij}^{(+,-)}=\left(q_{ij}^{(+,+)}\right)\inv=\left(q_{ij}^{(-,-)}\right)\inv=q_{ij}^{(-,+)}.\end{equation}
We simplify notation and write $q_{ij}=q_{ij}^{(+,+)}=q_{ij}^{(-,-)}$, with $q_{ij}\inv=q_{ij}^{(+,-)}=q_{ij}^{(-,+)}$.

For each $(\alpha,\beta)\in \ZZ^m\times \ZZ^n$ we define $\gt_{(\alpha,\beta)}\in \Aut_{\kk}(A\otimes B)$ as follows:
\[\gt_{(\alpha,\beta)}(r\otimes s)=r\otimes s, ~ \gt_{(\alpha,\beta)}(X_j^{\pm}\otimes 1)=X_j^{\pm}\otimes 1,~\gt_{(\alpha,\beta)}(1\otimes Y_i^{\pm})=\prod_{j=1}^m (q_{ij}^{\pm 1})^{\alpha_j}  (1\otimes Y_i^{\pm}),\]
for all $r\in R$, $s\in S$, $j\in\{1,\ldots,m\}$, $i\in\{1,\ldots,n\}$.

It is clear that $\gt_{(\alpha+\alpha',\beta+\beta')}=\gt_{(\alpha,\beta)}\gt_{(\alpha',\beta')}$. Therefore $\gt:\Gamma\to \Aut_{\kk}(A\otimes B)$ is a homomorphism and $\gt=\{\gt_{(\alpha,\beta)}~|~(\alpha,\beta)\in\ZZ^m\otimes\ZZ^n\}$ is a twisting system on $A\otimes B$, considered as a $\Gamma$-graded algebra with $\Gamma=\ZZ^m\otimes \ZZ^n$.

\begin{proposition}\label{prop.twtw}
Assume that $A$ and $B$ are TGWAs satisfying Hypothesis \ref{hyp.tgwa}, where the automorphisms $\sigma$ and $\rho$ are extended trivially. Let $\gt$ be as above and $\tau$ as in \eqref{eq.TGWCtau}. Then $\lt(A \tensor B) \iso A \tensor_\tau B$, with the isomorphism given by $\id_{A\otimes B}$.
\end{proposition}
\begin{proof}
We use $\star_\gt$ (resp. $*_\tau$) to denote the multiplication in $\lt(A \tensor B)$ (resp. $A \tensor_\tau B$). It suffices to prove that, for any pair of elements $x,y\in A\otimes B$, $x\star_\gt y=x*_\tau y$, and by linearity it is enough to check it for homogeneous elements. Let $\bu,\bu' \in (\pm[m])^k$, $\bv, \bv' \in (\pm[n])^\ell$, $r,r' \in R$, and $s,s' \in S$. Then
\begin{align*}
    (rX_\bu &\tensor sY_\bv) *_\tau (r'X_{\bu'} \tensor s'Y_{\bv'} )
    = q_{(\bv,\bu')} r\sigma^{d(\bu)}(r')X_{\bu\bu'} \tensor s\rho^{d(\bv)}(s')Y_{\bv  \bv'} \\
    &= q_{(\bv,\bu')} \left(rX_\bu \tensor sY_\bv\right)\left(r'X_{\bu'} \tensor s'Y_{\bv'}\right)
    = \gt_{(\bu',\bv')}\left(rX_\bu \tensor sY_\bv\right)(r'X_{\bu'} \tensor s'Y_{\bv'}) \\
        &= (rX_\bu \tensor sY_\bv) \star_\chi (r'X_{\bu'} \tensor s'Y_{\bv'}).\qedhere
\end{align*}
%The result follows.
\end{proof}

%\begin{align*}
%     \phi \left( (rX_\alpha \tensor sY_\beta) * (r'X_{\alpha'} \tensor s'Y_{\beta'} ) \right)
%        &= \phi( \ct((0,\beta),(\alpha',0)) r\sigma^\alpha(r')X_{\alpha+\alpha'} \tensor s\rho^\beta(s')Y_{\beta + \beta'} ) \\
%        &= \ct(\beta,\alpha') r\sigma^\alpha(r')X_{\alpha+\alpha'} \tensor s\rho^\beta(s')Y_{\beta + \beta'} \\
%        &= q_{(\beta,\alpha')} r\sigma^\alpha(r')X_{\alpha+\alpha'} \tensor s\rho^\beta(s')Y_{\beta + \beta'} \\
%        &= (rX_\alpha \tensor sY_\beta) \star (r'X_{\alpha'} \tensor s'Y_{\beta'}).
%\end{align*}

\section{Modules over twists}
%\daniele{($m$ is already the rank of $A$, so we should have a different notation for elements of $M$, maybe $x$ or $w$?)}
We conclude by discussing Proposition \ref{prop.twtw} in the context of graded modules over twisted tensor products and graded twists of TGWAs. 
%We use this to describe the graded simple modules over a family of TGWAs of type $(A_1)^n$, which may be obtained as graded twists of the $n$th Weyl algebra $A_n(\kk)$, equivalently, twisted tensor products of copies of the first Weyl algebra $A_1(\kk)$. 
Unless otherwise stated, all modules are left modules.

Let $\Gamma$ be an additive abelian group with identity $\bzero$. If $A$ is a $\Gamma$-graded algebra, then an $A$-module $M$ is \emph{graded} if $M$ has a vector space decomposition $M=\bigoplus_{\gamma \in \Gamma} M_\gamma$ such that $A_\alpha M_\gamma = M_{\gamma+\alpha}$. Given a graded module $M$ and $\alpha \in \Gamma$, the \emph{graded shift} $M\langle \alpha \rangle$ is a graded module $M\langle \alpha \rangle=\bigoplus_{\gamma \in \Gamma} M\langle \alpha \rangle_\gamma$ where $M\langle \alpha \rangle_\gamma=M_{\gamma-\alpha}$. The category $\Gr_\Gamma A$ consists of graded $A$-modules and homomorphisms $\phi:M \to N$ satisfying $\phi(M_\gamma) \subset N_{\gamma}$ for all $\gamma \in \Gamma$. 

Let $M\in\Gr_\Gamma A$, and let $\chi$ be a twisting system of $A$, then we define $\lt{M}\in\Gr_\Gamma \ltA$ to be the same as $M$ as a vector space, with action given by $a \cdot w = \gt_\gamma(a)w$ for $a \in \ltA$, $w \in M_\gamma$. This gives an equivalence of categories between $\Gr_\Gamma A$ and $\Gr_\Gamma \ltA$, which is the identity on vector spaces and on morphisms, see \cite[Theorem 3.1]{ztwist}.
%If $B$ is a $\Gamma$-graded twist of $A$, say $B \simeq \ltA$, then there is an equivalence between $\Gr_\Gamma A$ and $\Gr_\Gamma B$ which is the identity at the level of vector spaces \cite[Theorem 3.1]{ztwist}. The action of $\ltA$ on a graded $A$-module $M$ is given by $a \cdot m = \gt_\beta(a)m$ for $a \in A_\alpha$, $m \in M_\beta$. A morphism $\psi:M \to N$ in $\Gr_\Gamma A$ is also a morphism $\psi':\lt{M} \to \lt{N}$ in $\Gr_\Gamma B$ and $\psi=\psi'$ as vector-space maps.

\subsection{Modules over twisted tensor products}

Here we assume Hypothesis \ref{hyp.tgwa}. Let $M$ be a graded $A$-module and $N$ a graded $B$-module. Then $M \tensor N$ is naturally a graded $(A \tensor B)$-module. Following \cite{cap}, we show how we may consider extending this structure to the twisted tensor products defined above.

%Moreover, if $A$ is $\ZZ^m$-graded and $B$ is $\ZZ^n$-graded, and if $M$ and $N$ are graded modules, then $M \tensor N$ is naturally a graded module over the $(\ZZ^{m+n})$-graded algebra $A \tensor B$. 

Let $\tau:B \tensor A \to A \tensor B$ be as in \eqref{eq.TGWCtau}. Let $\lambda_A$ and $\lambda_B$ denote the actions of $A$ and $B$ on $M$ and $N$, respectively. An \emph{exchange map} for the pair $(B,M)$ is a $\kk$-linear map $\tau_M: B \tensor M \to M \tensor B$ such that $\tau_M(1_B \tensor w) = w \tensor 1_B$ for all $w \in M$. We can define an action $\tau_{\lambda_M}: A \tensor B \tensor M \tensor N \to M \tensor N$ by
\[ \tau_{\lambda_M} = (\lambda_A \tensor \lambda_B) \circ (\id_A \tensor \tau_M \tensor N).\]
Then $M \tensor N$ is an $(A \tensor_\tau B)$-module if and only if the following diagram commutes:
\[ 
\adjustbox{scale=.95,center}{%
\begin{tikzcd}
& B \tensor A \tensor B \tensor M \ar[r, "\tau \tensor \tau_M"] 
    & A \tensor B \tensor M \tensor B \ar[dr,"\id_A \tensor \tau_M \tensor \id_B"] \\
B \tensor B \tensor A \tensor M \arrow[ur, "\id_B \tensor \tau \tensor \id_M"]
    \arrow[r, swap, "\mu_B \tensor \lambda_A"]
& B \tensor M \arrow[r, "\tau_M"] & M \tensor B & A \tensor M \tensor B \tensor B
    \arrow[l, "\lambda_A \tensor \mu_B"]
\end{tikzcd}}
\]
This diagram matches the one for $\tau$ when $M=A$, $\lambda_A=\mu_A$, and $\tau_M = \tau$ (that is, by considering the action of $A$ on itself).
%The proof of the next proposition is analogous to Lemma \ref{lem.assoc}.
\begin{proposition}\label{prop.twmod}
Assume Hypothesis \ref{hyp.tgwa} where the automorphisms $\sigma$ and $\tau$ are extended trivially. Let $M$ and $N$ be graded modules over $A$ and $B$, respectively.
Define the scalars $q_{(\bv,\bu)}$ as in \eqref{eq.qdef}. 
For $w \in M_\alpha$ and $sY_\bv \in B$, set
\begin{align}\label{eq.taumod}
\tau_M(sY_\bv \tensor w) = q_{(\bv,\bu)} w \tensor sY_\bv.
\end{align}
Where $\bu\in(\pm [m])^k$ is any sequence such that $d(\bu)=\alpha$. Then $\tau_M$ is a twisting map and makes $M \tensor N$ into a $A \tensor_\tau B$ module.
\end{proposition}
\begin{proof}
The fact that the definition of $\tau_M$ does not depend on the choice of $\bu$ follows from \eqref{eq:qij-inv}. The rest of the proof is analogous to Lemma \ref{lem.assoc}. 
\end{proof}

\subsection{Relating graded module categories} 
Keep the notation of Proposition \ref{prop.twtw} and let $\Gamma=\ZZ^{m+n}$. Applying \cite[Theorem 3.1]{ztwist}, we have
\begin{align}\label{eq.grmod}
\Gr_\Gamma A \tensor B \equiv \Gr_\Gamma \lt(A \tensor B) \equiv \Gr_\Gamma A \tensor_\tau B.
\end{align}
The action in the case of the graded twist is the standard one and the corresponding action on the twisted tensor product is the one given in Proposition \ref{prop.twmod}.

\begin{example}\label{ex.weyl}
Suppose that $\kk$ is algebraically closed.
The $n$th Weyl algebra $A_n(\kk)$ is a TGWA of Cartan type $(A_1)^n$ over $R=\kk[t_1,\hdots,t_n]$ with $\sigma_i(t_j)=t_j-\delta_{ij}$ and $\mu_{ij}=1$ for all $i,j \in [n]$.

%Let $X_\alpha$ be a monomial in $A_n(\kk)$ for $\alpha \in \ZZ^n$. Up to scalar, we may write $X_\alpha$ as $X_1^{\alpha_1} \cdots X_n^{\alpha_n}$ where we interpret $X_i^{\alpha_i} = (X_i^+)^{\alpha_i}$ for $i\geq 0$ and $X_i^{\alpha_i} = (X_i^-)^{|\alpha_i|}$ for $i < 0$. 

If $\bu \in (\pm[n])^k$, then 
$[t_i,X_{\bu}]=d(\bu)_i X_{\bu}$ for any $i$. Thus, the grading on $A_n(\kk)$ is given by commuting with the $t_i$. Said another way, a graded left $A_n(\kk)$-module is a $(A_n(\kk),R)$-bimodule. 
The right $R$-module structure is given by $w \cdot r = \sigma^{\alpha}(r)w$ for any $r \in R$, $w \in M_\alpha$.
Hence, the graded $A_n(\kk)$-modules are those that decompose as a sum of weight subspaces over $R$. By \cite[Section 5.3]{BBF}, up to ungraded isomorphism, the simple $A_n(\kk)$-weight modules are isomorphic to a tensor product of $A_1(\kk)$-weight modules.

Let $N$ be a simple weight module of $A_n(\kk)$, so $N \iso N_1 \tensor \hdots \tensor N_n$ for $A_1(\kk)$-weight modules $N_i$. The graded isomorphism classes of $N$ are obtained by shifts. So if $\alpha = (\alpha_i) \in \ZZ^n$, then $N\langle \alpha \rangle \iso N_1\langle \alpha_1 \rangle \tensor \cdots \tensor N_n\langle \alpha_n \rangle$. That is, up to isomorphism, the graded simple $A_n(\kk)$-modules are tensor products of graded simple $A_1(\kk)$-modules. 
Define the $A=A_1(\kk)$-modules
\[ U=A/AX_1^+, \quad V=(A/AX_1^-)\langle 1 \rangle, \quad M_\lambda = A/A(t_1+\lambda)\] 
for $\lambda \in \kk\backslash\ZZ$. By \cite[Lemma 4.1]{sierra}, the graded simple $A_1(\kk)$-modules are (up to isomorphism) the $M_\lambda$ along with the shifts $U\langle n \rangle$ and $V\langle n\rangle$ for $n \in \ZZ$.

By an inductive version of \eqref{eq.grmod}, this also describes the graded simple modules of graded twists of $A_n(\kk)$, as well as graded twisted tensor products of copies of $A_1(\kk)$ where the automorphisms are always extended trivially. In this way one obtains the simple graded modules of the algebras in Example \ref{ex.mpwa}.
\end{example}

\section*{Acknowledgment}
D.R. was supported in this work by a Grant-in-Aid of Research and a Summer Faculty Fellowship from Indiana University Northwest. The authors thank Robert Won for helpful conversations related to this project.

%\bibliography{tgwabib}{}
%\bibliographystyle{myplain}

\end{document}